\documentclass[12pt,]{article}

\usepackage[margin=1in]{geometry}
\usepackage{amscd}
\usepackage{amssymb}
\usepackage{amsmath}
\usepackage{amsthm}
\usepackage{enumerate}
\usepackage{xspace}
\usepackage{tikz}
\usetikzlibrary{arrows}
\newenvironment{customlegend}[1][]{%
    \begingroup
    \csname pgfplots@init@cleared@structures\endcsname
    \pgfplotsset{#1}%
}{%
    \csname pgfplots@createlegend\endcsname
    \endgroup
}%

\def\addlegendimage{\csname pgfplots@addlegendimage\endcsname}

\usepackage[bb=dsserif]{mathalpha}
\usepackage{pgfplots}
\pgfplotsset{compat=1.16}
\usepackage{wrapfig, framed, caption}
\usepackage{float}
\usepackage[font=small,labelfont=bf]{caption}
\usepackage{xcolor}
\usepackage{mathtools}
\usepackage[colorlinks=true, linkcolor=blue, citecolor=blue,
pagebackref=true]{hyperref}
\usepackage{ulem}
\usepackage{comment}

\usepackage[capitalise]{cleveref}

\newtheorem{theorem}{Theorem}[section]
\newtheorem*{theorem*}{Theorem}

\newtheorem*{theoremY*}{Theorem Y}

\newtheorem*{theoremAB*}{Theorem AB}

\newtheorem{corollary}[theorem]{Corollary}
\newtheorem*{corollary*}{Corollary}
\newtheorem{proposition}[theorem]{Proposition}
\newtheorem{lemma}[theorem]{Lemma}

\newtheorem*{claim*}{Claim}

\theoremstyle{definition}
\newtheorem{definition}[theorem]{Definition}
\theoremstyle{remark}

\newtheorem*{remark*}{Remark}

\renewcommand{\Bbb}[1]{\mathbb{#1}}

\newcommand{\bbE}{{\Bbb E}}

\newcommand{\bbN}{{\Bbb N}}         

\newcommand{\bbP}{{\Bbb P}}
\newcommand{\bbQ}{{\Bbb Q}}         
\newcommand{\bbR}{{\Bbb R}}        

\newcommand{\bbZ}{{\Bbb Z}}         

\newcommand{\cC}{{\cal C}}

\newcommand{\cF}{{\cal F}}
\newcommand{\cG}{{\cal G}}
\newcommand{\cH}{{\cal H}}
\newcommand{\cI}{{\cal I}}

\newcommand{\cQ}{{\cal Q}}

\newcommand{\cV}{{\cal V}}

\newcommand{\ba}{{\overline a}}
\newcommand{\eg}{\textit{e.g.}\@\xspace}
\newcommand{\ie}{\textit{i.e.}\@\xspace}
\newcommand{\diam}{\text{diam}}

\renewcommand{\le}{\leq}
\renewcommand{\ge}{\geq}

\newcommand{\bi}{{\overline {\imath}}}
\newcommand{\bj}{{\overline  {\jmath}}}

\newcommand{\bo}{{\overline o}}

\DeclareMathOperator{\dimh}{\dim_H}

\allowdisplaybreaks

\title{Dynamical covering sets in self-similar sets}

\author{Balazs Barany\footnote{BB acknowledges support from the grants NKFI~K142169 and NKFI
  KKP144059 ``Fractal geometry and applications'' Research Group.} \\(BME) \and Henna
  Koivusalo\footnote{HK acknowledges travel funding through A\"OU's collaborative grant 103\"ou6.}
  \\ (Bristol) \and Sascha Troscheit\footnote{ST was initially supported by FWF Lise Meitner
    Fellowship M-2813 and European Research Council Marie Sk\l{}odowska--Curie Personal Fellowship
    \#101064701. ST also acknowledges travel funding through A\"OU's collaborative grant 103\"ou6 and
Lisa \& Carl-Gustav Esseens fund for mathematics (2025).} \\(Uppsala)}
\date{\today}

\newcommand{\Addresses}{{
    \bigskip
    \footnotesize

    B.~B\'ar\'any, \textsc{Department of Stochastics, HUN-REN-BME Stochastics Research Group,
      Institute of Mathematics, Budapest University of Technology and Economics, M\H{u}egyetem rkp.~3.,
    H-1111 Budapest, Hungary.}\par\nopagebreak
    \textit{E-mail address:} \texttt{barany.balazs@ttk.bme.hu}

    \medskip

    H.~Koivusalo, \textsc{School of Mathematics, University of Bristol, Fry Building, Woodland Road,
    Bristol, BS8 1UG, United Kingdom.}\par\nopagebreak
    \textit{E-mail address:} \texttt{henna.koivusalo@bristol.ac.uk}
    \medskip

    S.~Troscheit, \textsc{Department of Mathematics, Box 480, 751 06 Uppsala University, Sweden.}\par\nopagebreak
    \textit{E-mail address:} \texttt{sascha.troscheit@math.uu.se}
}}

\begin{document}

\frenchspacing
\maketitle

\begin{abstract}
  We study the size of \emph{dynamical covering sets} on a self-similar set. Dynamical covering sets
  are limsup sets generated by placing shrinking target sets around points along an orbit in a
  dynamical system.  In the case when the target sets are balls with sizes depending on the centre,
  we determine the size of the dynamical covering set as a function of the shrinking rate. In
  particular, we find sharp conditions guaranteeing when full Dvoretzky-type covering, and full
  measure occur. We also compute the Hausdorff dimension in the remaining cases.  The proofs apply
  in the cases of targets centred at typical points of the self-similar set, with
  respect to any  Bernoulli measure on it. Unlike in existing work on dynamical coverings, and
  despite the dimension value featuring phase transitions, we demonstrate that the behaviour can be
  characterised by a single pressure function over the full range of parameters. The techniques are
  a combination of classical dimension theoretical estimates and intricate martingale arguments.
\end{abstract}

%
%

\section{Introduction}

An
{\it iterated function system (IFS)} is a collection of contractive maps $\{f_1, \dots, f_N\}$, and
is associated with a unique, non-empty, compact attractor, or invariant set
$\Lambda=\cup_{i=1}^Nf_i(\Lambda)$. Iterated function systems are one of the  main methods for
generating sets with fractal features. The invariance of the set $\Lambda$ under the iterated
function system implies an amount of scale-invariance that guarantees fine structure on all levels
of magnification. The modern theory of iterated function systems began with Hutchinson
\cite{Hutchinson1981}, who computed the Hausdorff dimension of $\Lambda$ in the case of similarity
maps $\{f_1, \dots, f_N\}$, satisfying a separation condition. In this case, the set $\Lambda$ is
said to be {\it self-similar}. 

There is a natural expanding dynamical system $T:\Lambda\to \Lambda$ given locally by the inverses
of $f_i$. More precisely, for any  point $x\in  f_i(\Lambda)$, one can define $T(x)$ by
$f_i^{-1}(x)$ for an appropriate choice of $i$. Such $i$ must exist, but is not necessarily unique.
Under sufficient separation conditions, the choice of $f_i$ is always unique and the dynamical system
is topologically conjugate to the left shift map on $N$ symbols, via a coding map of points of $\Lambda$. 

In the context of dynamical systems, a classical question is that of \emph{recurrence}. A particular
example is the \emph{shrinking target problem}: How large is the set of points that hit a sequence
of targets $A_i$ infinitely many times? That is, how large is the set
\[
  R_{st} = \{x \in X : T^n(x) \in A_n \text{ for infinitely many }n\in\bbN\}.
\]
The term shrinking target was introduced by Hill and Velani, see \cite{HillVelani1995,
HillVelani1999}. Questions on such shrinking targets for iterated function systems have received an
increasing amount of interest over the past decade, see \cite{AllenBakerBarany2025, AllenBarany2021,
Baker2023, Baker2024, BakerKoivusalo2024, BaranyRams2018, BaranyTroscheit2022, Daviaud2025,
KirsebomKundePersson2023, KoivusaloRamirez2018, Seuret2018} and the references therein.  A
particular feature of such sets is that they can be interpreted as $\limsup$ sets, since
\[
  R_{st} = \limsup_{n\in\bbN} T^{-n}(A_n) = \bigcap_{n\in\bbN}\bigcup_{k\geq n}T^{-k}(A_k).
\]

Instead of a sequence of targets, one can consider the orbit of a (typical) point $x_0\in X$ under
the dynamics $T:X\to X$ and study the set of points that lie in $T^n(x_0)+A_n$ infinitely often, where
$A_n$ is a sequence of nested sets. For example, when $A_n = B(0,r_n)$ is a
sequence of balls, we set
\[
  R(x_0, A_n) = \limsup_{n\in\bbN} T^n(x_0)+A_n = \limsup_{n\in\bbN} B(T^n(x_0),r_n)
  =\bigcap_{n\in\bbN}\bigcup_{k\geq n} B(T^k(x_0),r_k).
\]
The question on whether $R(x_0, A_n) = X$ is known as the \emph{dynamical covering problem} and is
the dynamical analogue of the classical Dvoretzky covering problem from probability theory (see
\eg \cite{Dvoretzky1956} and, more recently, \cite{JarvenpaaJarvenpaaMyllyojaStenflo2025, JMS25}).  When
$R(x_0, A_n)\subsetneq X$ one can ask the finer question of the size of $R(x_0, A_n)$ in terms of
measure or dimension. This problem has been studied in some restricted set-ups for interval maps,
see \cite{FanSchmelingTroubetzkoy2013, LiaoSeuret2013} and also \cite{PerssonRams2017}. (In these
works, the dynamical covering problem is referred to as dynamical Diophantine approximation, or
shrinking targets problem). In all of these works it is observed that the size of the dynamical
covering goes through distinct regimes with several phase transitions.  Namely, when $r_n \approx
n^{-1/\alpha}$ and $\alpha$ is small, the dimension is linear in $\alpha$, but the system goes
through a phase transition when $\alpha$ equals the entropy of the measure, after which it follows a
multifractal spectrum. One of the key techniques in these works is the mass transference principle,
that gives a straightforward method to lower dimension bound estimates. 

In the present article, we study the dynamical covering set in the context of self-similar sets
satisfying the open set condition. For definitions see \cref{sec:background}. For a typical centre
(with respect to a Bernoulli measure on $\Lambda$), and for target sizes given by symbolic balls
shrinking at a rate associated with the centre point, we find sharp conditions under which the
dynamical covering set is all of $\Lambda$. This is the analogue of the Dvoretzky covering problem
in the study of random covering sets. Furthermore, we find sharp conditions under which the
dynamical covering set is of full measure with respect to the Hausdorff measure on $\Lambda$ (the
Hausdorff measure $\mathcal H^{\dim \Lambda}$ is positive and finite due to the open set condition,
see Hutchinson \cite{Hutchinson1981}). The bulk of the work in this article, however, is to compute
the Hausdorff dimension of the dynamical covering set in all the remaining cases. We also compute
the dimension of the complement of the dynamical covering set,
for those parameter values when the dynamical covering set is of full measure. 

To complement the past results on the topic, we present a single pressure formula governing the
dimension value across the full range of shrinking rates. The phase transitions show as alternative
forms to this pressure formula, see  \cref{thm:legendre}.  This article is the first to study
self-similar sets with inhomogeneous contraction rates, which reveals novel phenomena: We observe
that for small values of $\alpha$, the dimension value is no longer linear in $\alpha$, but strictly
concave. Furthermore, the dimension behaviour for larger values of $\alpha$ is not given by the
multifractal spectrum of the associated measure, but through a different variational principle that
we investigate.  Further, in case of symbolic balls, the inhomogeneity of the contraction ratios
makes it impossible to use the method of mass transference principle that underlays all previous
results. Our set-up necessitates the use of intricate martingale methods to determine the dimension.
The situation is analogous to a result from the shrinking target context \cite{AllenBarany2021}, which
also highlight that the mass transference principle is ineffectual when the contractions are
inhomogeneous and the balls are symbolic.

\subsection{Main result} \label{sec:background}

Let us present more precise definitions before giving the full statements of the main theorems. Let
$\Phi=\{f_1, \dots, f_N\}$ be an iterated function system of similarities on $\bbR^d$, with
contraction ratios $\lambda_i\in(0,1)$. Denote by $\Lambda$ the corresponding self-similar set,
satisfying \(\Lambda=\bigcup_{i=1}^N f_i(\Lambda).\) We assume throughout that the iterated function
system satisfies the Open Set Condition (OSC), that is, there exists an open and bounded set
$U\subset\bbR^d$ such that $f_i(U)\subset U$ for every $i=1,\ldots,N$ and $f_i(U)\cap
f_j(U)=\emptyset$ for every $i\neq j$. We assume without loss of generality that $\Lambda\cap
U\neq\emptyset$, see Bandt and Graf \cite{BandtGraf1992} and Schief \cite{Schief1994}. Denote by
$s_0$ the Hausdorff dimension of $\Lambda$, which is the unique solution of
\begin{equation*}
  \sum_{i=1}^N\lambda_i^{s_0}=1,
\end{equation*}
see Hutchinson \cite{Hutchinson1981}. For simplicity, let
$\lambda_{\max}=\max\{\lambda_i:i\in\{1,\ldots,N\}\}<1$ and
$\lambda_{\min}=\min\{\lambda_i:i\in\{1,\ldots,N\}\}>0$.

Denote the symbolic space by $\Sigma=\{1, \dots, N\}^\bbN$, with the left shift map $\sigma$. We
write $\Sigma_n$ for the collection of words of length $n$, and $\Sigma_*=\cup_{n=0}^\infty
\Sigma_n$. For $\bi=(i_1, \dots, i_n)\in \Sigma_n$, let $f_{\bi}=f_{i_1}\circ\cdots\circ f_{i_n}$
and $\lambda_\bi=\lambda_{i_1}\cdots \lambda_{i_n}$. Let $\bi|_n$ denote the first $n$ digits of a
word $\bi\in\Sigma\cup\Sigma_*$, and $|\bi|$ the length of a finite word. Denote by $\bi<\bj$ when a
finite word $\bi$ appears as the initial segment of a finite or infinite $\bj$. Define cylinders for
any $\bi\in \Sigma_*$ by \([\bi]=\{\bj\in \Sigma\mid \bi<\bj\}.\) 
The symbolic space gives a coding of points of $\Lambda$ via the map $\pi: \Sigma\to \Lambda$,
\[
  \pi(\bi)=\lim_{n\to \infty} f_{\bi|_n}(0).
\]
Note that the map $\pi\colon\Sigma\to\Lambda$ is ``almost'' one-to-one, namely, the points with
multiple preimages in $\Sigma$ are negligible in some proper sense.

\begin{definition}
Let $\bo\in \Sigma$ and let $\ell\colon\bbN\to \bbN$ be any increasing function diverging to
$\infty$. The {\bf dynamical covering set} $R(\bo, \ell)\subseteq \Lambda$ is
\begin{equation*}
  R(\bo, \ell)=\pi\left(\left\{\bi\in \Sigma\mid \bi\in\left[(\sigma^n\bo)|_{\ell(n)}\right]\text{
	infinitely
  often}\right\}\right)=\pi\left(\bigcap_{k=1}^\infty\bigcup_{n=k}^\infty\left[(\sigma^n\bo)|_{\ell(n)}\right]\right).
\end{equation*}
\end{definition}

We can also consider the map $T\colon\Lambda\to\Lambda$ that takes any
$x\in f_i(\Lambda)$ to $f_i^{-1}(x)\in \Lambda$. This is an expanding map on the
points of unique coding, which conjugates to $\sigma$ by $\pi\circ \sigma=T\circ \pi$. At points
with multiple codings, a choice of the map $f_i^{-1}$ has to be made.  Notice that the set $R(\bo,
\ell)$ is analogous to the dynamical covering set on $\Lambda$ following the orbit of the
point $\pi(\bo)$ under $T$, namely, to the set
\[
  \widehat R(\bo, \ell)=\left\{x\in \Lambda\mid x\in B\left(T^n(\pi(\bo)),
  \lambda_{(\sigma^n\bo)|_{\ell(n)}}\right)\text{ infinitely often}\right\}.
\]
However, in what follows, we will focus our analysis on the projected symbolic dynamical covering set $R(\bo, \ell)$. 

Let $p=(p_1, \dots, p_N)$ be a probability vector, and $\bbP_p$ the corresponding Bernoulli measure
on $\Sigma$.  Without loss of generality, we may assume that $p_i>0$ for every $i=1,\ldots,N$. For
simplicity, we denote the push-forward measure $\pi_*\bbP_p$ on $\Lambda$ also by $\bbP_p$.
Let us denote the largest and the smallest probabilities by
$p_{\max}=\max\{p_i:i\in\{1,\ldots,N\}\}$ and $p_{\min}=\min\{p_i:i\in\{1,\ldots,N\}\}$.
Our main theorem on the dynamical covering set states

\begin{theorem}\label{thm:main}
  Let $\{f_1, \dots, f_N\}$ be an IFS of similarities satisfying the open set condition, and with
  contraction ratios $\lambda_1, \dots, \lambda_N$. Let $\alpha>0$ and assume that $\ell: \mathbb
  N\to \mathbb N$ is monotone increasing and it satisfies
  \begin{equation}\label{eq:condonl}
    \liminf_{n\to \infty} \frac{\ell(n)}{\log n}=\tfrac 1\alpha.
  \end{equation}
  Let $\bbP_p$ be a Bernoulli measure corresponding to a probability vector $p=(p_1, \dots, p_N)$.
  Then, for $\bbP_{p}$-almost every $\bo \in \Sigma$, the Hausdorff dimension of $R(\bo, \ell)$ is
  given by the solution $s=s(\alpha)$ to
  \begin{equation}\label{eq:salpha}
    s(\alpha)=\inf_{q\in[0,1]} \left\{P_\alpha(q) : \sum_{i=1}^N \lambda_i^{P_{\alpha}(q)}(p_ie^\alpha)^q=1\right\}.
  \end{equation}

  Furthermore, we have for $\bbP_{p}$-almost every $\bo$ that if $-\log p_{\min}>\alpha>-
  \sum_{i=1}^N \lambda_i^{s_0}\log  p_i$ then $\cH^{s_0}(R(\bo, \ell))=\cH^{s_0}(\Lambda)$ but
  $R(\bo, \ell)\neq\Lambda$, and if $\alpha>-\log p_{\min}$ then $R(\bo, \ell)=\Lambda$.
\end{theorem}

We remark that the statement of \cref{thm:main} remains valid for the set $\widehat R(\bo,\ell)$.
Indeed, we may assume without loss of generality that the diameter $\diam(\Lambda)=1$, and so,
$R(\bo,\ell)\subseteq\widehat R(\bo,\ell)$ for every $\bo\in\Sigma$, which gives the lower bound. On
the other hand, a similar argument to the one in Section~\ref{sec:upper} with only minor
modification gives the upper bound. We leave the proof to the interested reader.

Our second main result is on the size of the complement of the dynamical covering set:
\begin{theorem}\label{thm:main2}
  Let $\{f_1, \dots, f_N\}$ be an IFS of similarities satisfying the open set condition, and with
  contraction ratios $\lambda_1, \dots, \lambda_N$. Let $\alpha>0$ and assume that $\ell: \mathbb
  N\to \mathbb N$ is monotone increasing and it satisfies
  \begin{equation*}
    \liminf_{n\to \infty} \frac{\ell(n)}{\log n}=\tfrac 1\alpha.
  \end{equation*}
  Let $\bbP_p$ be a Bernoulli measure corresponding to a probability vector $p=(p_1, \dots, p_N)$.
  Then, for $\bbP_{p}$-almost every $\bo \in \Sigma$, the Hausdorff dimension of $R(\bo, \ell)^c$ is
  given by the solution $t=t(\alpha)$ to
  \begin{equation}\label{eq:salpha2}
    t(\alpha)=\inf_{q<0} \left\{P_\alpha(q) : \sum_{i=1}^N \lambda_i^{P_{\alpha}(q)}(p_ie^\alpha)^q=1\right\}.
  \end{equation}
  Furthermore, if $\alpha<-
  \sum_{i=1}^N \lambda_i^{s_0}\log  p_i$ then $\cH^{s_0}(R(\bo, \ell)^c)=\cH^{s_0}(\Lambda)$ and if
  $\alpha>-\log p_{\min}$ then $R(\bo,\ell)^c=\emptyset$ for $\bbP_{p}$-almost every $\bo$.
\end{theorem}

The proofs fall into, broadly, two separate categories. For large $\alpha$, as described in
\cref{sec:upper,sec:lowerlarge}, we rely on  covering arguments and the Borel-Cantelli lemma, as
well as the description of the multifractal spectrum, provided to us by a probabilistic pressure
formula (see \cref{eq:probpressure} underneath). For small $\alpha$, in \cref{sec:lowersmall}, the
proof techniques are much more involved, and require sophisticated arguments involving martingale
and large deviation estimates.

\subsection{Discussion for different values of \texorpdfstring{$\alpha$}{alpha}}\label{sec:discuss}
We now discuss the results of \cref{thm:main} and \cref{thm:main2} in more detail. In
particular, the purpose of this section is to make sense of the dimension formula and present the
phase transitions that the size of the dynamical covering set goes through as $\alpha$ varies. We
identify the four distinct regions of $\alpha$ for the different sizes of $R(\bo,\ell)$, and the
corresponding behaviour of $R(\bo, \ell)^c$. Furthermore, at the end of this subsection we present
figures showing examples of the possible behaviours of $R(\bo, \ell)$ as a function of $\alpha$. 

Let $P_\alpha(q)$ be the
unique solution of
$$
\sum_{i=1}^N \lambda_i^{P_{\alpha}(q)}(p_ie^\alpha)^q=1.
$$
We refer to \cref{sec:comparison} for a proof of the existence of $P_\alpha(q)$. We will show that
as $\alpha$ increases, the behaviour of the sizes of the sets $R(\bo, \ell)$ and $R(\bo, \ell)^{c}$
goes through phase transition at the values 
\[
  \alpha_0=-\sum_{i=1}^N\lambda_i^{P_{\alpha_0}(1)}p_ie^{\alpha_0} \log p_i, \quad
  \alpha_1=-\sum_i\lambda_i^{s_0}\log p_i, \quad \text{and} \quad \alpha_2=-\log p_{\min}. 
\]
The first region corresponds to ``small'' values of $\alpha\in (0, \alpha_0]$ and corresponds to the
``pressure part''. Namely, there is a unique value $\alpha_0\in(0,\infty)$ such that
$\alpha_0=\allowbreak-\sum_{i=1}^N\lambda_i^{P_{\alpha_0}(1)}p_ie^{\alpha_0} \log p_i$ and
$s(\alpha)=P_\alpha(1)$ holds for every $\alpha\in(0,\alpha_0]$. In particular, if 
$$
\alpha<-\sum_{i=1}^N\lambda_i^{P_\alpha(1)}p_ie^\alpha \log p_i
\quad\text{ then }\quad
s(\alpha)=P_\alpha(1).
$$
In the special case when
$\lambda_i\equiv\lambda$, we have that $\alpha_0=-\sum_ip_i\log p_i$ and $s(\alpha)=\frac{\alpha}{-\log\lambda}$ for
$\alpha\in(0,\alpha_0]$.

The second region corresponds to the ``spectral behaviour part''. More precisely, for every
$\alpha\in[\alpha_0,\alpha_1]$, there exists a unique $q=q(\alpha)\in[0,1]$
such that almost surely,
$$
\dimh R(\bo,\ell)=P_{\alpha}(q(\alpha))=\inf_{q\in\bbR}P_\alpha(q)=\sup\left\{\frac{-\sum_iq_i\log
q_i}{-\sum_iq_i\log\lambda_i}:-\sum_iq_i\log p_i\leq\alpha\right\}.
$$ This optimisation resembles the multifractal spectrum, and it indeed coincides with the Legendre
transform of the $L^q$-dimension of $\bbP_p$ if $\lambda_i\equiv\lambda$, which is the special
dynamical case considered by Liao and Seuret
\cite{LiaoSeuret2013}. For $\alpha\in (0, \alpha_1]$, the dimension satisfies
$\dim_HR(\bo,\ell)<s_0$. We note
here that for all values $\alpha<\alpha_1$, the complementary set $R(\bo, \ell)^c$ has full
$\cH^{s_0}$-measure in $\Lambda$. 

The next two regions are the ``large'' $\alpha$ cases when typical orbits cover
$\cH^{s_0}$-almost every point in $\Lambda$. However, in the region $\alpha\in(\alpha_1, \alpha_2)$,
almost surely the whole set $\Lambda$ is not covered,
whereas if $\alpha>\alpha_2$ then every point in $\Lambda$ is covered $\bbP_p$-almost surely
infinitely often. This solves the analogue of Dvoretzky's covering problem in the self-similar
setting. In the region $\alpha_2>\alpha>\alpha_1$, the
Hausdorff measure of the complementary set $R(\bo,\ell)^c$ is zero and, in particular, its Hausdorff
dimension is given by the ``decreasing'' part of the spectrum corresponding to negative $q$ values
$$
\dimh R(\bo,\ell)^c=\inf_{q<0}P_\alpha(q)=\inf_{q\in\bbR}P_\alpha(q)=\sup\left\{\frac{-\sum_iq_i\log
q_i}{-\sum_iq_i\log\lambda_i}:-\sum_iq_i\log p_i\geq\alpha\right\}
$$
almost surely. It is strictly smaller than $s_0$, see \cref{sec:upperforcomp,sec:lbver1}. Once
$\alpha_2<\alpha$, the set $R(\bo,\ell)^c$ is clearly empty.

Surprisingly, the $s(\alpha)$ defined in \cref{eq:salpha} has an equivalent characterisation as the unique root of
\begin{equation}\label{eq:probpres}
  \limsup_{n\to \infty} \tfrac 1n \log \sum_{|\bj|=n}\lambda_{\bj}^{s}\left(1-(1-p_{\bj})^{e^{\alpha n}}\right)=0,
\end{equation}
see \cref{sec:comparison}. This alternative form of $s(\alpha)$ plays an important role in the proof
of \cref{thm:main}: It is more geometric in nature and gives a natural description of coverings by
taking into account the multiplicity of the visits of cylinders by typical $\bbP_p$-orbits.

In the figures underneath we depict the possible behaviours of the sizes of $R(\bo, \ell)$ and
$R(\bo,\ell)^c$ as $\alpha$ varies, for different types of  iterated function systems. In the first figure,
\cref{fig:ex2}, we present a generic image describing the size of the  dynamical covering set
$R(\bo, \ell)$. The dimension value passes from convex (green) to concave (yellow) at $\alpha_0$, eventually ending
in a plateau (red) at the dimension of the whole self-similar set $\Lambda$ after $\alpha_1$. The figures
also show the final phase transition where $R(\bo, \ell)$ becomes all of $\Lambda$ at $\alpha_2$
(dotted to solid line).
The dimension value of its complement $R(\bo,\ell)^c$ is shown in black and behaves dually to the
dimension of $R(\bo,\ell)$, being maximal until $\alpha_1$ after which it follows the concave
region.

In \cref{fig:ex1}, we present the Hausdorff dimensions of $R(\bo, \ell)$ and $R(\bo,\ell)^c$ in the case when the
distribution of $\bo$ follows the dimension maximising measure. Note that the second phase
transition $\alpha_1$ is the smallest by using the dimension maximising measure. In other words,
it is the case where the dynamical covering set covers the full Hausdorff measure set with the fastest speed.

In the example presented in \cref{fig:homogeneous}, unlike in the other examples depicted, the shape
of the dimension spectrum is the same as in the work of \cite{LiaoSeuret2013}. This is due to the
fact that in this example, the IFS is homogeneous, which implies that the radii of the balls in the
dynamical covering do not depend on the orbit $\bo$. In particular, the dimension of $R(\bo, \ell)$
is linear in $\alpha$ for $\alpha<\alpha_0$.

The last figure, \cref{fig:degenerate}, shows that it is possible for the three critical values of $\alpha$ to
coincide, \ie $\alpha_0=\alpha_1=\alpha_2$. In this example, the dimension value is given by
$P_\alpha(1)$ and furthermore, as soon as $\dim_H R(\bo, \ell)$ is equal to the dimension of
$\Lambda$, also $R(\bo, \ell)=\Lambda$. This is due to the fact that the Bernoulli measure has equal
probabilities. Clearly, the phase transition $\alpha_2$ is the smallest for the uniform
distribution, \ie the dynamical covering set covers the whole set with the fastest speed.
Conversely, the dimension of $R(\bo,\ell)^c$ is maximal until the phase transition after which
$R(\bo,\ell)^c = \varnothing$.

Finally, we briefly comment on the nature of the phase transitions $\alpha_0$ and $\alpha_1$. Since
the functions $g_1:\alpha\mapsto P_\alpha(1)$ and $g_2:\alpha\mapsto \inf_{q\in\bbR} P_\alpha(q)$
are real analytic functions which are convex
and concave, respectively, we have $g_1(\alpha)\geq g_2(\alpha)$ whenever defined, and the phase transition occurs at
$\alpha_0$, when both functions coincide, we must necessarily have that their derivatives agree at
$\alpha_0$.
Again by convexity and concavity, the second derivatives are strictly positive and negative,
respectively, and the phase transition is of second order. The only exception is when the Legendre
transform is trivial and a singleton (and $\alpha_0 = \alpha_1 = \alpha_2$). In this case the derivative of the
first function is strictly positive, whereas the second region is of derivative zero. Here we obtain
a first order phase transition.

\definecolor{col1}{HTML}{20a39e}
\definecolor{col2}{HTML}{ffba49}
\definecolor{col3}{HTML}{ef5b5b}
\definecolor{col4}{HTML}{888888}
\definecolor{col5}{HTML}{000088}

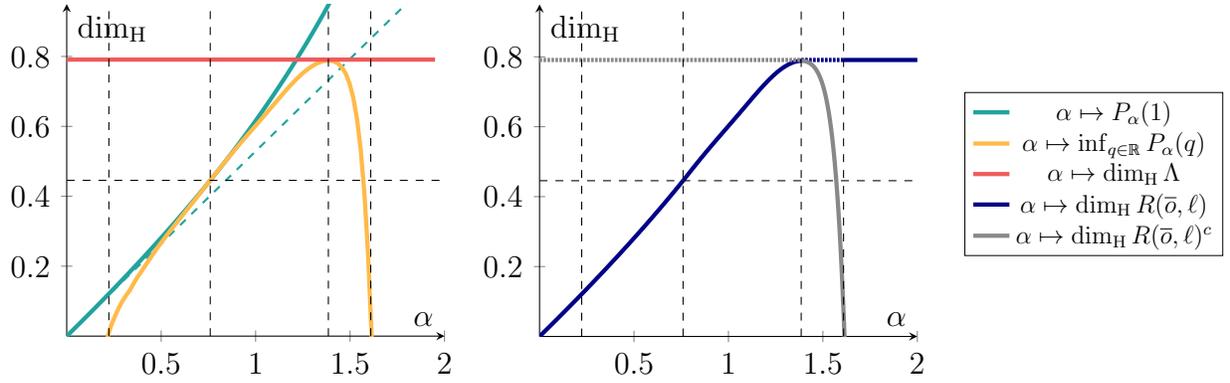
\begin{figure}[p] 
  \begin{center}
    \begin{tikzpicture}
      \begin{axis}[
	domain=0:1.95,
	samples=100,
	axis lines=middle,
	xlabel={$\alpha$}, ylabel={$\dimh$},
	xmin=0, xmax=2, ymin=0, ymax=0.95,
	width=0.4\textwidth, height=6cm
	]

	\pgfmathsetmacro{\sZero}{0.791002}
	\pgfmathsetmacro{\sOne}{0.445581}
	\pgfmathsetmacro{\tOne}{1.38513}
	\pgfmathsetmacro{\tTwo}{1.60944}
	\pgfmathsetmacro{\tThree}{0.223144}
	\pgfmathsetmacro{\tFour}{0.759745}

	\addplot[col1, ultra thick] {0.0000429021 + 0.528817 *x + 0.0595684 *x^2 + 0.00187359* x^3 +
	  0.0258721* x^4};
	\addplot[col1, thick, dashed] { 0.528817 *x };
	\addplot[col2, ultra thick] {0.788466 + 0.00611217* (x - 1.38513) - 2.69672* (x - 1.38513)^2 -
	    10.6174* (x - 1.38513)^3 - 65.167* (x - 1.38513)^4 -
	    292.99 *(x - 1.38513)^5 - 754.17* (x - 1.38513)^6 -
	    1123.49 *(x - 1.38513)^7 - 965.325* (x - 1.38513)^8 -
	  444.938* (x - 1.38513)^9 - 85.2819* (x - 1.38513)^10};
	\addplot[col3,ultra thick]{\sZero};
	\addplot[dashed] coordinates {(\tOne,0) (\tOne,1)};
	\addplot[dashed] coordinates {(\tTwo,0) (\tTwo,1)};
	\addplot[dashed] coordinates {(\tThree,0) (\tThree,1)};
	\addplot[dashed] coordinates {(\tFour,0) (\tFour,1)};
	\addplot[dashed] coordinates {(0,\sOne) (2,\sOne)};
      \end{axis}
    \end{tikzpicture}
    \begin{tikzpicture}
      \begin{axis}[
	domain=0:1.7,
	samples=100,
	axis lines=middle,
	xlabel={$\alpha$}, ylabel={$\dimh$},
	xmin=0, xmax=2, ymin=0, ymax=0.95,
	width=0.4\textwidth, height=6cm
	]

	\pgfmathsetmacro{\sZero}{0.791002}
	\pgfmathsetmacro{\sOne}{0.445581}
	\pgfmathsetmacro{\tOne}{1.38513}
	\pgfmathsetmacro{\tTwo}{1.60944}
	\pgfmathsetmacro{\tThree}{0.223144}
	\pgfmathsetmacro{\tFour}{0.759745}

	\addplot[col5, ultra thick,domain=0:\tFour] {0.0000429021 + 0.528817 *x + 0.0595684 *x^2 + 0.00187359* x^3 +
	  0.0258721* x^4};
	\addplot[col5, ultra thick,domain=\tFour:\tOne] {0.788466 + 0.00611217* (x - 1.38513) -
	    2.69672* (x - 1.38513)^2 -
	    10.6174* (x - 1.38513)^3 - 65.167* (x - 1.38513)^4 -
	    292.99 *(x - 1.38513)^5 - 754.17* (x - 1.38513)^6 -
	    1123.49 *(x - 1.38513)^7 - 965.325* (x - 1.38513)^8 -
	  444.938* (x - 1.38513)^9 - 85.2819* (x - 1.38513)^10};
	\addplot[col5,ultra thick,dash pattern=on 1pt off .5pt,domain=\tOne:\tTwo]{\sZero};
	\addplot[col5,ultra thick,domain=\tTwo:2]{\sZero};

	\addplot[col4, ultra thick,domain=0:\tOne,dash pattern=on 1pt off .5pt] {\sZero};
	\addplot[col4, ultra thick,domain=\tOne:\tTwo+0.1] {0.788466 + 0.00611217* (x - 1.38513) -
	    2.69672* (x - 1.38513)^2 -
	    10.6174* (x - 1.38513)^3 - 65.167* (x - 1.38513)^4 -
	    292.99 *(x - 1.38513)^5 - 754.17* (x - 1.38513)^6 -
	    1123.49 *(x - 1.38513)^7 - 965.325* (x - 1.38513)^8 -
	  444.938* (x - 1.38513)^9 - 85.2819* (x - 1.38513)^10};

	\addplot[dashed] coordinates {(\tOne,0) (\tOne,1)};
	\addplot[dashed] coordinates {(\tTwo,0) (\tTwo,1)};
	\addplot[dashed] coordinates {(\tThree,0) (\tThree,1)};
	\addplot[dashed] coordinates {(\tFour,0) (\tFour,1)};
	\addplot[dashed] coordinates {(0,\sOne) (2,\sOne)};

      \end{axis}
    \end{tikzpicture}
    \raisebox{1.7cm}{
      \scalebox{0.8}{
	\begin{tikzpicture}
	  \begin{customlegend}[legend entries={$\alpha\mapsto P_\alpha(1)$,$\alpha\mapsto
	      \inf_{q\in\bbR}P_\alpha(q)$,$\alpha\mapsto\dimh\Lambda$,{$\alpha\mapsto\dimh
	    R(\bo,\ell)$},{$\alpha\mapsto\dimh R(\bo,\ell)^c$}}]
	    \addlegendimage{col1,line width=2 pt,sharp plot}
	    \addlegendimage{col2,line width=2 pt,sharp plot}
	    \addlegendimage{col3,line width=2 pt,sharp plot}
	    \addlegendimage{col5,line width=2 pt,sharp plot}
	    \addlegendimage{col4,line width=2 pt,sharp plot}
	  \end{customlegend}
	\end{tikzpicture}
      }
    }
    \caption{Example with IFS consisting of two maps, $f_1:x\mapsto 0.8x$
    and $f_2: x\mapsto 0.1x+0.9$. The probabilities are $p_1=0.2$ and $p_2=0.8$. }\label{fig:ex2}
  \end{center}
\end{figure}

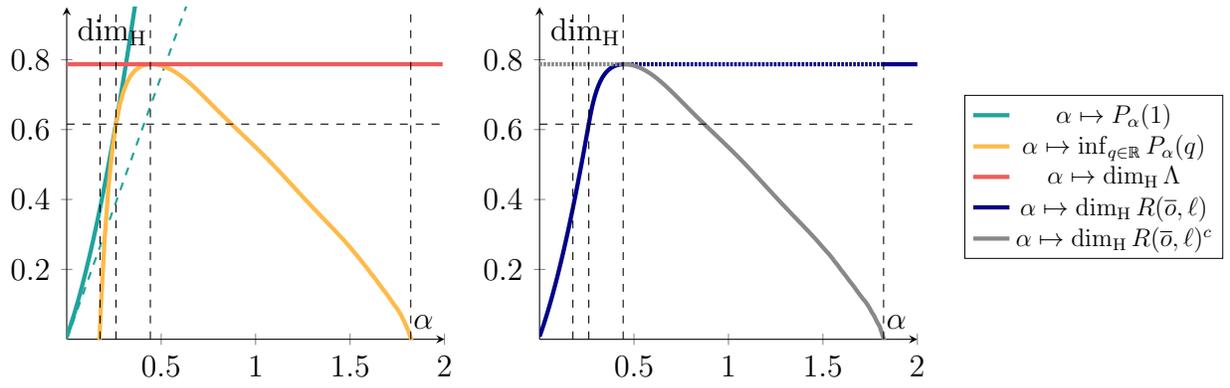
\begin{figure}[p]
  \begin{center}
    \begin{tikzpicture}
      \begin{axis}[
	domain=0:1.99,
	samples=100,
	axis lines=middle,
	xlabel={$\alpha$}, ylabel={$\dimh$},
	xmin=0, xmax=2, ymin=0, ymax=0.95,
	width=0.4\textwidth, height=6cm
	]
	\pgfmathsetmacro{\sZero}{0.786971}
	\pgfmathsetmacro{\sOne}{0.615567}
	\pgfmathsetmacro{\tOne}{0.442654}
	\pgfmathsetmacro{\tTwo}{1.82135}
	\pgfmathsetmacro{\tThree}{0.176507}
	\pgfmathsetmacro{\tFour}{0.260725}

	\addplot[col1, ultra thick] {0.00774617 + 1.50563 *x + 3.45146* x^2 - 0.910021* x^3 - 0.377264* x^4};
	\addplot[col1, thick, dashed] { 1.50563 *x };
	\addplot[col2, ultra thick] {0.786971 + 0.0210774* (-0.442654 + x) - 1.7915* (-0.442654 + x)^2 +
	    4.93679 *(-0.442654 + x)^3 - 32.0515* (-0.442654 + x)^4 +
	    141.756 *(-0.442654 + x)^5 - 329.497* (-0.442654 + x)^6 +
	    428.456 *(-0.442654 + x)^7 - 316.514* (-0.442654 + x)^8 +
	  124.451 *(-0.442654 + x)^9 - 20.2559* (-0.442654 + x)^10};
	\addplot[col3,ultra thick]{\sZero};
	\addplot[dashed] coordinates {(\tOne,0) (\tOne,1)};
	\addplot[dashed] coordinates {(\tTwo,0) (\tTwo,1)};
	\addplot[dashed] coordinates {(\tThree,0) (\tThree,1)};
	\addplot[dashed] coordinates {(\tFour,0) (\tFour,1)};
	\addplot[dashed] coordinates {(0,\sOne) (2,\sOne)};
      \end{axis}
    \end{tikzpicture}
    \begin{tikzpicture}
      \begin{axis}[
	domain=0:1.7,
	samples=100,
	axis lines=middle,
	xlabel={$\alpha$}, ylabel={$\dimh$},
	xmin=0, xmax=2, ymin=0, ymax=0.95,
	width=0.40\textwidth, height=6cm
	]
	\pgfmathsetmacro{\sZero}{0.786971}
	\pgfmathsetmacro{\sOne}{0.615567}
	\pgfmathsetmacro{\tOne}{0.442654}
	\pgfmathsetmacro{\tTwo}{1.82135}
	\pgfmathsetmacro{\tThree}{0.176507}
	\pgfmathsetmacro{\tFour}{0.260725}

	\addplot[col5, ultra thick,domain=0:\tFour] {0.00774617 + 1.50563 *x + 3.45146* x^2 - 0.910021* x^3 - 0.377264* x^4};
	\addplot[col5, ultra thick,domain=\tFour:\tOne] {0.786971 + 0.0210774* (-0.442654 + x) - 1.7915* (-0.442654 + x)^2 +
	    4.93679 *(-0.442654 + x)^3 - 32.0515* (-0.442654 + x)^4 +
	    141.756 *(-0.442654 + x)^5 - 329.497* (-0.442654 + x)^6 +
	    428.456 *(-0.442654 + x)^7 - 316.514* (-0.442654 + x)^8 +
	  124.451 *(-0.442654 + x)^9 - 20.2559* (-0.442654 + x)^10};
	\addplot[col5,ultra thick,dash pattern=on 1pt off .5pt,domain=\tOne:\tTwo]{\sZero};
	\addplot[col5,ultra thick,domain=\tTwo:2]{\sZero};

	\addplot[col4, ultra thick,domain=0:\tOne,dash pattern=on 1pt off .5pt] {\sZero};
	\addplot[col4, ultra thick,domain=\tOne:\tTwo+0.1] {0.786971 + 0.0210774* (-0.442654 + x) - 1.7915* (-0.442654 + x)^2 +
	    4.93679 *(-0.442654 + x)^3 - 32.0515* (-0.442654 + x)^4 +
	    141.756 *(-0.442654 + x)^5 - 329.497* (-0.442654 + x)^6 +
	    428.456 *(-0.442654 + x)^7 - 316.514* (-0.442654 + x)^8 +
	  124.451 *(-0.442654 + x)^9 - 20.2559* (-0.442654 + x)^10};

	\addplot[dashed] coordinates {(\tOne,0) (\tOne,1)};
	\addplot[dashed] coordinates {(\tTwo,0) (\tTwo,1)};
	\addplot[dashed] coordinates {(\tThree,0) (\tThree,1)};
	\addplot[dashed] coordinates {(\tFour,0) (\tFour,1)};
	\addplot[dashed] coordinates {(0,\sOne) (2,\sOne)};
      \end{axis}
    \end{tikzpicture}
        \raisebox{1.7cm}{
      \scalebox{0.8}{
	\begin{tikzpicture}
	  \begin{customlegend}[legend entries={$\alpha\mapsto P_\alpha(1)$,$\alpha\mapsto
	      \inf_{q\in\bbR}P_\alpha(q)$,$\alpha\mapsto\dimh\Lambda$,{$\alpha\mapsto\dimh
	    R(\bo,\ell)$},{$\alpha\mapsto\dimh R(\bo,\ell)^c$}}]
	    \addlegendimage{col1,line width=2 pt,sharp plot}
	    \addlegendimage{col2,line width=2 pt,sharp plot}
	    \addlegendimage{col3,line width=2 pt,sharp plot}
	    \addlegendimage{col5,line width=2 pt,sharp plot}
	    \addlegendimage{col4,line width=2 pt,sharp plot}
	  \end{customlegend}
	\end{tikzpicture}
      }
    }
    \caption{Example with IFS $f_1: x \mapsto 0.8x$ and $f_2: x \mapsto 0.1x+0.9$. The probabilities
      are $p_1=0.8^s\approx 0.838$ and $p_2=0.1^s\approx 0.162$,
      where $s=\dimh \Lambda$ (the dimension
    maximising measure).}\label{fig:ex1}
  \end{center}
\end{figure}
\begin{figure}[p]
  \begin{center}
    \begin{tikzpicture}
      \begin{axis}[
	domain=0:1.95,
	samples=100,
	axis lines=middle,
	xlabel={$\alpha$}, ylabel={$\dimh$},
	xmin=0, xmax=2, ymin=0, ymax=1.0,
	width=0.4\textwidth, height=6cm
	]
	\pgfmathsetmacro{\sZero}{0.7565}
	\pgfmathsetmacro{\sOne}{0.6}
	\pgfmathsetmacro{\tOne}{0.9163}
	\pgfmathsetmacro{\tTwo}{1.55}
	\pgfmathsetmacro{\tThree}{0.3}
	\pgfmathsetmacro{\tFour}{0.6/1.06}

	\addplot[col1, ultra thick] {1.06 *x};
	\addplot[col2, ultra thick] {0.756864 - 0.000250806  *(x - 0.9163) - 1.18476 * (x - 0.9163)^2 -
	    0.00875389* (x - 0.9163)^3 - 0.688954* (x - 0.9163)^4 -
	    0.0594986 *(x - 0.9163)^5 - 8.56599 *(x - 0.9163)^6 +
	    0.136719 *(x - 0.9163)^7 + 25.0002 *(x - 0.9163)^8 -
	  0.0918132 *(x - 0.9163)^9 - 28.1464*(x - 0.9163)^10};

	\addplot[col3,ultra thick]{\sZero};
	\addplot[dashed] coordinates {(\tOne,0) (\tOne,1)};
	\addplot[dashed] coordinates {(\tTwo,0) (\tTwo,1)};
	\addplot[dashed] coordinates {(\tThree,0) (\tThree,1)};
	\addplot[dashed] coordinates {(\tFour,0) (\tFour,1)};
	\addplot[dashed] coordinates {(0,\sOne) (2,\sOne)};
      \end{axis}
    \end{tikzpicture}
    \begin{tikzpicture}
      \begin{axis}[
	domain=0:1.7,
	samples=100,
	axis lines=middle,
	xlabel={$\alpha$}, ylabel={$\dimh$},
	xmin=0, xmax=2, ymin=0, ymax=1.00,
	width=0.4\textwidth, height=6cm
	]
	\pgfmathsetmacro{\sZero}{0.7565}
	\pgfmathsetmacro{\sOne}{0.6}
	\pgfmathsetmacro{\tOne}{0.9163}
	\pgfmathsetmacro{\tTwo}{1.55}
	\pgfmathsetmacro{\tThree}{0.3}
	\pgfmathsetmacro{\tFour}{0.6/1.06}

	\addplot[col5, ultra thick,domain=0:\tFour] {1.06 *x};
	\addplot[col5, ultra thick,domain=\tFour:\tOne] {0.756864 - 0.000250806  *(x - 0.9163) -
	    1.18476 * (x - 0.9163)^2 -
	    0.00875389* (x - 0.9163)^3 - 0.688954* (x - 0.9163)^4 -
	    0.0594986 *(x - 0.9163)^5 - 8.56599 *(x - 0.9163)^6 +
	    0.136719 *(x - 0.9163)^7 + 25.0002 *(x - 0.9163)^8 -
	  0.0918132 *(x - 0.9163)^9 - 28.1464*(x - 0.9163)^10};
	\addplot[col5,ultra thick,dash pattern=on 1pt off .5pt,domain=\tOne:\tTwo]{\sZero};
	\addplot[col5,ultra thick,domain=\tTwo:2]{\sZero};

	\addplot[col4, ultra thick,domain=0:\tOne,dash pattern=on 1pt off .5pt] {\sZero};
	\addplot[col4, ultra thick,domain=\tOne:\tTwo+0.1] {0.756864 - 0.000250806  *(x - 0.9163) -
	    1.18476 * (x - 0.9163)^2 -
	    0.00875389* (x - 0.9163)^3 - 0.688954* (x - 0.9163)^4 -
	    0.0594986 *(x - 0.9163)^5 - 8.56599 *(x - 0.9163)^6 +
	    0.136719 *(x - 0.9163)^7 + 25.0002 *(x - 0.9163)^8 -
	  0.0918132 *(x - 0.9163)^9 - 28.1464*(x - 0.9163)^10};

	\addplot[dashed] coordinates {(\tOne,0) (\tOne,1)};
	\addplot[dashed] coordinates {(\tTwo,0) (\tTwo,1)};
	\addplot[dashed] coordinates {(\tThree,0) (\tThree,1)};
	\addplot[dashed] coordinates {(\tFour,0) (\tFour,1)};
	\addplot[dashed] coordinates {(0,\sOne) (2,\sOne)};
      \end{axis}
    \end{tikzpicture}
            \raisebox{1.7cm}{
      \scalebox{0.8}{
	\begin{tikzpicture}
	  \begin{customlegend}[legend entries={$\alpha\mapsto P_\alpha(1)$,$\alpha\mapsto
	      \inf_{q\in\bbR}P_\alpha(q)$,$\alpha\mapsto\dimh\Lambda$,{$\alpha\mapsto\dimh
	    R(\bo,\ell)$},{$\alpha\mapsto\dimh R(\bo,\ell)^c$}}]
	    \addlegendimage{col1,line width=2 pt,sharp plot}
	    \addlegendimage{col2,line width=2 pt,sharp plot}
	    \addlegendimage{col3,line width=2 pt,sharp plot}
	    \addlegendimage{col5,line width=2 pt,sharp plot}
	    \addlegendimage{col4,line width=2 pt,sharp plot}
	  \end{customlegend}
	\end{tikzpicture}
      }
    }
    \caption{Example with IFS $f_1:x\mapsto 0.4 x$ and $f_2: x\mapsto
      0.4x+0.6$. The probabilities are $p_1=0.2$ and $p_2=0.8$. The case
      coincides with the result of Liao and Seuret \cite{LiaoSeuret2013}.}
      \label{fig:homogeneous}
  \end{center}
\end{figure}
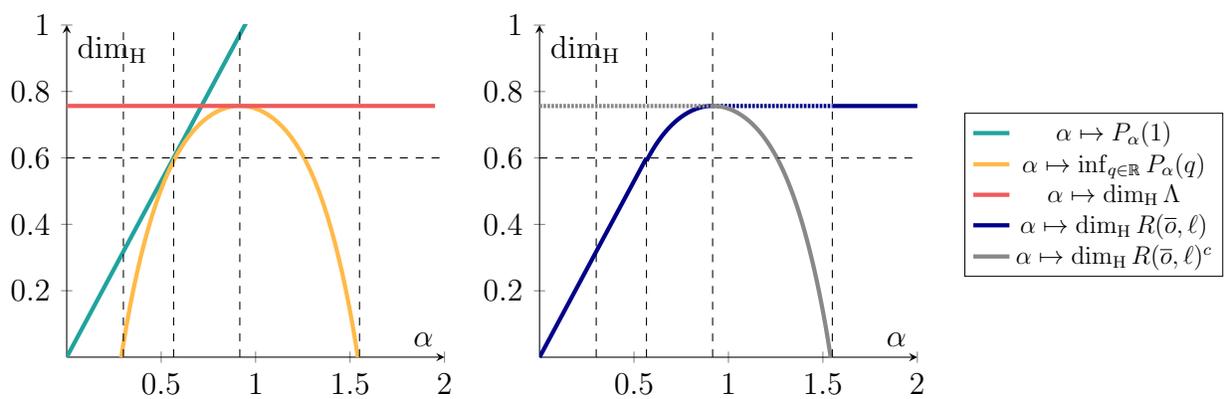

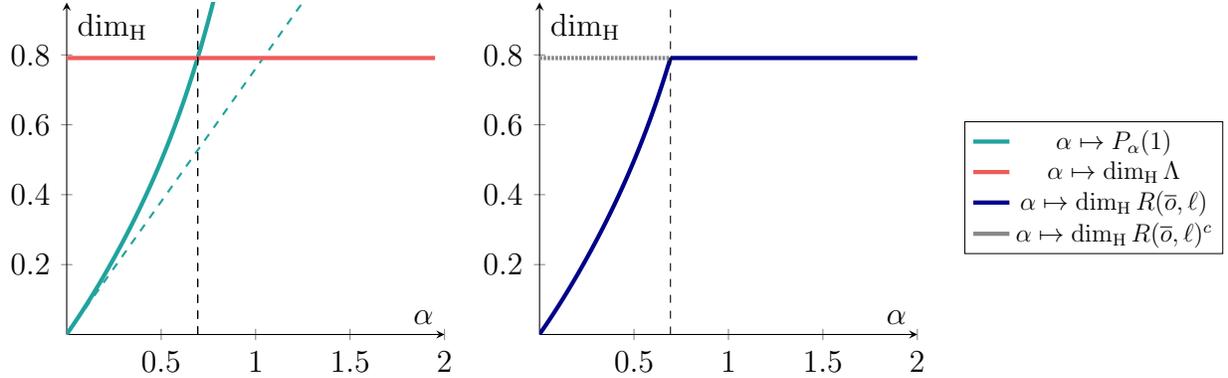
\begin{figure}[pt]
  \begin{center}
    \begin{tikzpicture}
      \begin{axis}[
	domain=0:1.95,
	samples=100,
	axis lines=middle,
	xlabel={$\alpha$}, ylabel={$\dimh$},
	xmin=0, xmax=2, ymin=0, ymax=0.95,
	width=0.4\textwidth, height=6cm
	]
	\pgfmathsetmacro{\sZero}{0.791002}
	\pgfmathsetmacro{\sOne}{0.791002}
	\pgfmathsetmacro{\tOne}{0.693147}
	\pgfmathsetmacro{\tTwo}{0.693147}
	\pgfmathsetmacro{\tThree}{0.693147}
	\pgfmathsetmacro{\tFour}{0.693147}

	\addplot[col1, ultra thick] {0.00101662 + 0.760572 *x + 0.496951* x^2 - 0.425069* x^3 + 0.717325* x^4};
	\addplot[col1, thick, dashed] { 0.760572 *x };
	\addplot[col3,ultra thick]{\sZero};
	\addplot[dashed] coordinates {(\tOne,0) (\tOne,1)};
	\addplot[dashed] coordinates {(\tTwo,0) (\tTwo,1)};
	\addplot[dashed] coordinates {(\tThree,0) (\tThree,1)};
	\addplot[dashed] coordinates {(\tFour,0) (\tFour,1)};
      \end{axis}
    \end{tikzpicture}
    \begin{tikzpicture}
      \begin{axis}[
	domain=0:1.7,
	samples=100,
	axis lines=middle,
	xlabel={$\alpha$}, ylabel={$\dimh$},
	xmin=0, xmax=2, ymin=0, ymax=0.95,
	width=0.4\textwidth, height=6cm
	]
	\pgfmathsetmacro{\sZero}{0.791002}
	\pgfmathsetmacro{\sOne}{0.791002}
	\pgfmathsetmacro{\tOne}{0.693147}
	\pgfmathsetmacro{\tTwo}{0.693147}
	\pgfmathsetmacro{\tThree}{0.693147}
	\pgfmathsetmacro{\tFour}{0.693147}

	\addplot[col5, ultra thick,domain=0:\tFour] {0.00101662 + 0.760572 *x + 0.496951* x^2 -
	  0.425069* x^3 + 0.717325* x^4};

	\addplot[col5,ultra thick,domain=\tTwo:2]{\sZero};

	\addplot[col4, ultra thick,domain=0:\tOne,dash pattern=on 1pt off .5pt] {\sZero};

	\addplot[dashed] coordinates {(\tOne,0) (\tOne,1)};
      \end{axis}
    \end{tikzpicture}
            \raisebox{1.7cm}{
      \scalebox{0.8}{
	\begin{tikzpicture}
	  \begin{customlegend}[legend entries={$\alpha\mapsto P_\alpha(1)$,$\alpha\mapsto\dimh\Lambda$,{$\alpha\mapsto\dimh
	    R(\bo,\ell)$},{$\alpha\mapsto\dimh R(\bo,\ell)^c$}}]
	    \addlegendimage{col1,line width=2 pt,sharp plot}
	    \addlegendimage{col3,line width=2 pt,sharp plot}
	    \addlegendimage{col5,line width=2 pt,sharp plot}
	    \addlegendimage{col4,line width=2 pt,sharp plot}
	  \end{customlegend}
	\end{tikzpicture}
      }
    }
    \caption{Example with IFS $f_1:x \mapsto 0.8x$ and $f_2: x \mapsto 0.1
      x+0.9$. The probabilities are equal $p_1=0.5=p_2$.
    }\label{fig:degenerate}
  \end{center}
\end{figure}

\subsection{Structure of the paper}

In \cref{sec:comparison} we prove that the dimension formula for $R(\bo, \ell)$ can be computed from
the probabilistic pressure \cref{eq:probpressure}, for all values of $\alpha$. This probabilistic
pressure is more closely related to the geometry of the set $R(\bo, \ell)$ than the function
$P_\alpha(q)$ and as such, is an important tool for obtaining the upper bound results in
\cref{sec:upper}. They are also key to the lower bound arguments in the low $\alpha$ regime,
$\alpha<\alpha_0$. In \cref{sec:upper}, we compute the dimension upper bounds for the dynamical
covering set and its complement, proving the upper bounds in \cref{thm:main,thm:main2}. These are
probabilistic covering arguments relying on the Borel-Cantelli lemma. 

In \cref{sec:lowerlarge}, we prove that 
$R(\bo,\ell)=\Lambda$ for almost all $\bo$ when $\alpha>\alpha_2$, see \cref{thm:fullCover}, that
$\mathcal H^{s_0}(R(\bo, \ell))=\mathcal H^{s_0}(\Lambda)$ for almost all $\bo$ when
$\alpha_1<\alpha<\alpha_2$, and that $\mathcal H^{s_0}(R(\bo, \ell)^c)=\mathcal H^{s_0}(\Lambda)$
for almost all $\bo$ when $\alpha_1>\alpha$.  This section also contains the proofs of the lower
bounds to $\dim_H R(\bo, \ell)$ for large $\alpha>\alpha_0$ and for $\dim_HR(\bo, \ell)^c$ for
$\alpha>\alpha_1$. The arguments in \cref{sec:lowerlarge} are similar to each other and to those in
\cref{sec:upper}, with the additional analysis of certain Bernoulli measures on $R(\bo, \ell)$ and
its complement.

The remainder of the paper is dedicated to the intricate proof of the lower bound to $\dim_H R(\bo, \ell)$ in
the case of small $\alpha\le\alpha_0$. This proof is contained in \cref{sec:lowersmall}. This
section begins with a construction of a Cantor set that is utilised for the lower bound proof, see
\cref{sec:cantor}. On this Cantor set, we construct a random measure, using martingale methods, see
\cref{sec:randommeasure}. We go on to prove that the random measure has some nice distribution
properties, see \cref{subsec:measure}. Finally, in \cref{subsec:energy} we complete the lower bound
proof through an energy estimate. The final short section \cref{sec:complete} proves the main
theorems by combining the results from the previous sections.

\paragraph{Acknowledgements:}
The authors would like to thank Amlan Banaji and Alex Rutar for discussions on the phase transitions
of the pressures.

\section{Preliminaries and the study of the probabilistic pressure}\label{sec:comparison}

This section focusses on analysis of the dimension value in \cref{thm:main,thm:main2}. In
particular, we show in \cref{thm:varprinc} that the number $P_{\alpha}(q)$ satisfies a variational
principle, and in \cref{thm:twospec} that it connects to the probabilistic pressure
\cref{eq:probpressure} defined underneath. These connections are crucial in the dimension estimates
of the later sections. 

We will repeatedly make use of the following well-known approximations throughout the whole paper.
For every $x\in\bbR$ and $n\in\bbN$
\begin{equation}\label{eq:repeat}
  (1-x)^n\ge 1-nx \quad\text{ and }\quad (1-x)^n\le e^{-xn}.
\end{equation}
Further, from a second order Taylor approximation, for every $x>0$ and $n\in\bbN$, we also always have
\begin{equation}\label{eq:taylor}
  (1-x)^n\le 1-nx+\frac{n^2x^2}{2}=1-nx(1-\frac{nx}2).
\end{equation}

First we prove the following proposition, which describes the regions discussed in
\cref{sec:discuss}. Let $p=(p_1,\ldots,p_N)$ be a probability vector such that $p_i>0$ for all
$i=1,\ldots,N$.

\begin{lemma}\label{thm:almostlegendre}
  Let $P_\alpha(q)$ be the unique solution of $\sum_{i=1}^N
  \lambda_i^{P_{\alpha}(q)}(p_ie^\alpha)^q=1.$ Then the map $q\mapsto P_\alpha(q)$ is continuous and convex.

  Moreover, if $p_i\neq p_j$ for some $i\neq j$ then the map
  \[
    \alpha\in[-\log p_{\max},-\log p_{\min}]\mapsto\inf_{q\in\bbR}P_\alpha(q)
  \]
  is also continuous, concave and its unique maximum is at $\alpha=- \sum_{i=1}^N \lambda_i^{s_0}\log p_i$.
\end{lemma}

\begin{proof}
  It is easy to see that the solution $P_\alpha(q)$ is indeed unique since the map
  \[
    x\mapsto\sum_{i=1}^N \lambda_i^{x}(p_ie^\alpha)^q
  \]
  is strictly monotone decreasing, continuous,
  goes to $0$ as $x\to\infty$ and goes to $\infty$ as $x\to-\infty$.

  It is clear from the definition that $q\mapsto P_\alpha(q)$ is continuous, so, we show that the
  map $q\mapsto P_{\alpha}(q)$ is convex (but not necessarily strictly convex) on $\bbR$. Simple
  algebraic calculations show that for every $q\in\bbR$
  \begin{equation}\label{eq:Pder}
    0=\sum_{i=1}^N\lambda_i^{P_\alpha(q)}(p_ie^\alpha)^q\left(\log p_i+\alpha+P_\alpha'(q)\log\lambda_i\right)
  \end{equation}
  and
  $$
  0=\sum_{i=1}^N\lambda_i^{P_\alpha(q)}(p_ie^\alpha)^q\left(\log
  p_i+\alpha+P_\alpha'(q)\log\lambda_i\right)^2+P_\alpha''(q)\sum_{i=1}^N\lambda_i^{P_\alpha(q)}(p_ie^\alpha)^q
  \log\lambda_i,
  $$
  thus,
  \[
    P_\alpha''(q)=\frac{\sum_{i=1}^N\lambda_i^{P_\alpha(q)}(p_ie^\alpha)^q\left(\log p_i+\alpha+P_\alpha'(q)\log\lambda_i\right)^2}{-\sum_{i=1}^N\lambda_i^{P_\alpha(q)}(p_ie^\alpha)^q \log\lambda_i}\geq0,
  \]
  hence, the convexity follows. Clearly, $P_\alpha''(q)=0$ for some $q\in\bbR$ if and only if
  $P_\alpha'(q)=\frac{\log(p_ie^\alpha)}{-\log\lambda_i}$ for every $i$ and every $q\in\bbR$. In
  particular, $P_\alpha''\equiv0$ and
  $\frac{\log(p_ie^\alpha)}{-\log\lambda_i}=\frac{\log(p_je^\alpha)}{-\log\lambda_j}$ for every
  $i\neq j$.

  Now, let us study the continuity of $\alpha\mapsto\inf_{q\in\bbR}P_\alpha(q)$. It is easy to see
  that if $p_{\min}e^\alpha>1$ then
  \[
    P_\alpha'(q)=\frac{\sum_{i=1}^N\lambda_i^{P_\alpha(q)}(p_ie^\alpha)^q\log(p_ie^\alpha)}{-\sum_{i=1}^N\lambda_i^{P_\alpha(q)}(p_ie^\alpha)^q\log\lambda_i}\geq\frac{\log(p_{\min}e^\alpha)}{-\log\min_i\{\lambda_i\}}>0
    \quad\text{ for every }
    q\in\bbR,
  \]
  and so $\inf_{q\in\bbR}P_\alpha(q)=\lim_{q\to-\infty}P_\alpha(q)=-\infty$. Similarly, if
  $p_{\max}e^\alpha<1$ then $P_\alpha'(q)<0$ for every $q\in\bbR$, and hence,
  $\inf_{q\in\bbR}P_\alpha(q)=\lim_{q\to\infty}P_\alpha(q)=-\infty$. Hence,
  $\alpha\mapsto\inf_{q\in\bbR}P_\alpha(q)$ defined only on $[-\log p_{\max},-\log p_{\min}]$.

  Let $\cI_{\max}=\{i:p_i=p_{\max}\}$ and $\cI_{\min}=\{i:p_i=p_{\min}\}$. Since
  $\alpha$ is an element in $[-\log p_{\max},-\log p_{\min}]$, for $q\geq0$ we get
  $$
  1=\sum_{i=1}^N\lambda_i^{P_\alpha(q)}(p_ie^\alpha)^q\geq\sum_{i\in \cI_{\max}}\lambda_i^{P_\alpha(q)}.
  $$
  Similarly, $1\geq\sum_{i\in \cI_{\min}}\lambda_i^{P_\alpha(q)}$ for $q\leq 0$. Thus,
  $P_\alpha(q)\geq0$ for every $q\in\bbR$. Let $s_{\min}$ and $s_{\max}$ be the unique solutions of
  $1=\sum_{i\in \cI_{\min}}\lambda_i^{s_{\min}}$ and $1=\sum_{i\in \cI_{\min}}\lambda_i^{s_{\max}}$,
  respectively.

  Under the assumption that $p_i\neq p_j$ for some $i\neq j$, a simple algebraic manipulation shows
  that if $\lambda_i=\lambda_j$ for every $i\neq j$ then there is no $\alpha$ such that
  $P_\alpha''\equiv0$. Moreover, if  $\lambda_i\neq\lambda_j$ for some $i\neq j$ then the only
  possible value of $\alpha$ for which $P_\alpha''\equiv0$ is
  $$
  \widehat{\alpha}=\frac{\tfrac{\log p_i}{-\log\lambda_i}-\tfrac{\log
  p_j}{-\log\lambda_j}}{\tfrac{1}{-\log\lambda_j}-\tfrac{1}{-\log\lambda_i}}.
  $$
  However, $\widehat{\alpha}\notin [-\log p_{\max},-\log p_{\min}]$. Indeed, if
  $\widehat{\alpha}\in [-\log p_{\max},-\log p_{\min}]$ then $P_{\widehat{\alpha}}'\equiv c$ for
  some $c\in\bbR$. If $P_{\widehat{\alpha}}'\equiv c\neq0$, then $P_{\widehat{\alpha}}(q)=s_0-cq$
  but then $\inf_{q\in\bbR}P_{\widehat{\alpha}}(q)=-\infty$, which contradicts to
  $P_{\widehat{\alpha}}(q)\geq0$ for every $q$, and if $c=0$ then $p_ie^{\widehat{\alpha}}=1$ for
  every $i$ but this contradicts to the assumption $p_i\neq p_j$ for some $i\neq j$. This implies
  that
  $$
  P_\alpha''(q)>0
  $$
  for every $\alpha\in [-\log p_{\max},-\log p_{\min}]$.

  Simple algebraic manipulation shows that
  $$
  \lim_{q\to\infty}P_\alpha'(q)=\frac{\sum_{i\in\cI_{\max}}\lambda_i^{s_{\max}}\log(p_i
  e^{\alpha})}{-\sum_{i\in\cI_{\max}}\lambda_i^{s_{\max}}\log\lambda_i}\geq 0\quad\text{and}\quad
  \lim_{q\to-\infty}P_\alpha'(q)=\frac{\sum_{i\in\cI_{\min}}\lambda_i^{s_{\min}}\log(p_i
  e^{\alpha})}{-\sum_{i\in\cI_{\min}}\lambda_i^{s_{\min}}\log\lambda_i}\leq 0.
  $$
  Hence, if $p_i\neq p_j$ for some $i\neq j$ then for every
  $\alpha\in(-\log p_{\max},-\log p_{\min})$ there exists a continuous function $\alpha\mapsto
  q(\alpha)$ such that $P_\alpha'(q(\alpha))=0$, and
  $\inf_{q\in\bbR}P_\alpha(q)=P_{\alpha}(q(\alpha))$. Moreover, the map $\alpha\mapsto
  P_{\alpha}(q(\alpha))$ is strictly concave,
  $\lim_{\alpha\to-\log p_{\max}}P_{\alpha}(q(\alpha))=s_{\max}$ and
  $\lim_{\alpha\to-\log p_{\min}}P_{\alpha}(q(\alpha))=s_{\min}$, which completes the proof.
\end{proof}

\cref{thm:almostlegendre} shows that $\inf_{q}P_\alpha(q)$ has similar properties to the standard
Legendre transform of convex functions.

\begin{proposition}\label{thm:legendre}
  Let $P_\alpha(q)$ be the unique solution of $\sum_{i=1}^N \lambda_i^{P_{\alpha}(q)}(p_ie^\alpha)^q=1.$ Then
  \begin{equation}\label{eq:regions}
    \inf_{q\in[0,1]} \left\{P_\alpha(q)\right\}=\begin{cases}
      P_{\alpha}(1), & \mbox{if } \alpha\leq-\sum_{i=1}^N\lambda_i^{P_\alpha(1)}p_ie^\alpha \log p_i \\
      \inf_{q\in\bbR} \left\{P_\alpha(q)\right\}, & \mbox{if } - \sum_{i=1}^N
      \lambda_i^{P_{\alpha}(0)}\log p_i>\alpha>-\sum_{i=1}^N\lambda_i^{P_{\alpha}(1)}p_ie^\alpha
      \log p_i \\
      P_{\alpha}(0)=s_0, & \mbox{otherwise}.
    \end{cases}
  \end{equation}
  Moreover, the map $\alpha\mapsto\inf_{q\in[0,1]}P_\alpha(q)$ is also continuous, and strictly
  smaller than $s_0=\dim_H\Lambda$ whenever $- \sum_{i=1}^N \lambda_i^{s_0}\log p_i>\alpha$.
\end{proposition}

\begin{proof}
  Let us first prove \cref{eq:regions}. Suppose that $P_\alpha''\equiv0$, and so
  there exists $c\in\bbR$ such that $\frac{\log(p_ie^\alpha)}{-\log\lambda_i}=c$ for every $i$. In
  particular, $P_\alpha(q)=cq+s_0$. Then
  $\alpha\leq-\sum_{i=1}^N\lambda_i^{P_\alpha(1)}p_ie^\alpha \log
  p_i=-\sum_{i=1}^N\lambda_i^{s_0}\log p_i$ holds if and only if $c\leq0$ and so
  $\inf_{q\in[0,1]}P_{\alpha}(q)=P_\alpha(1)$, which proves the claim in this case.

  We may suppose now that $P_\alpha''(q)>0$ for all $q\in\bbR$. Hence, there exists at most one
  $q'\in\bbR$ such that $P_{\alpha}'(q')=0$. Thus, if
  $\alpha\leq-\sum_{i=1}^N\lambda_i^{P_\alpha(1)}p_ie^\alpha \log p_i$ then $P_{\alpha}'(1)\leq0$ by
  \cref{eq:Pder}, and by convexity $1\leq q'$, and so, $P_{\alpha}'(q)\leq0$ for every $q\in[0,1]$,
  which clearly implies that $\inf_{q\in[0,1]} \left\{P_\alpha(q)\right\}=P_\alpha(1)$. The case
  $\alpha\geq-\sum_{i=1}^N\lambda_i^{P_\alpha(0)}\log p_i$ is similar, and the remaining case
  follows by the uniqueness of the infimal point of convex functions.
\end{proof}

Note that if $p_i=1/N$ for every $i$, then the middle region in \cref{eq:regions} is empty. Now, we
show the following variational principle result:

\begin{lemma}\label{thm:varprinc}
  For every $\alpha>0$,
  \begin{equation*}
    \sup\left\{\frac{-\sum_{i=1}^Nq_i\log q_i}{-\sum_{i=1}^Nq_i\log\lambda_i} :
    \sum_{i=1}^Nq_i\log p_i+\alpha\geq0\right\}=\inf_{q>0}P_\alpha(q),
  \end{equation*}
  and the supremum is attained for the probability vector
  $q_{i,\alpha}^*=\lambda_i^{P_\alpha(q^*)}(p_ie^{\alpha})^{q^*}$, where $q^*$ is where the
  infimum of $\inf_{q>0}P_\alpha(q)$ is attained.

  Moreover, for every $\alpha>0$,
  \begin{equation}\label{eq:varprinc2}
    \sup\left\{\frac{-\sum_{i=1}^Nq_i\log q_i}{-\sum_{i=1}^Nq_i\log\lambda_i} :
    \sum_{i=1}^Nq_i\log p_i+\alpha\leq0\right\}=\inf_{q<0}P_\alpha(q),
  \end{equation}
  and the supremum is attained for the probability vector
  $q_{i,\alpha}^*=\lambda_i^{P_\alpha(q^*)}(p_ie^{\alpha})^{q^*}$, where $q^*$ is where the
  infimum of $\inf_{q<0}P_\alpha(q)$ is attained.
\end{lemma}

\begin{proof}
  First, let us observe that by the concavity of the logarithm for every $q>0$ and for every
  probability vector $\{q_i\}_{i=1}^N$ such that $\sum_{i=1}^Nq_i\log p_i+\alpha\geq0$ we get
  \begin{align}
    0&=\log\sum_{i=1}^N\lambda_i^{P_\alpha(q)}(p_ie^{\alpha})^q
    \geq\sum_{i=1}^Nq_i\left(\log\left((p_ie^\alpha)^q\lambda_i^{P_\alpha(q)}\right)-\log
    q_i\right)\nonumber\\
     &=q\left(\sum_{i=1}^Nq_i\log p_i+\alpha\right)
     +P_{\alpha}(q)\sum_{i=1}^Nq_i\log\lambda_i-\sum_{i=1}^Nq_i\log q_i\nonumber\\
     &\geq P_{\alpha}(q)\sum_{i=1}^Nq_i\log\lambda_i-\sum_{i=1}^Nq_i\log q_i,\nonumber
  \end{align}
  and so,
  \begin{equation}\label{eq:corolary}
    P_\alpha(q)\geq\frac{-\sum_{i=1}^Nq_i\log q_i}{-\sum_{i=1}^Nq_i\log\lambda_i}.
  \end{equation}

  Now, suppose that $P_{\alpha}''(q)=0$ for every $q\in\bbR$, and in particular,
  $P_\alpha'(q)=\frac{\log p_i+\alpha}{-\log\lambda_i}=c$ for every $i$ and $q$. Then $c\geq0$ if
  and only if $\alpha\geq-\sum_i\lambda_i^{s_0}\log p_i$ and in this case
  $$
  \inf_{q>0}P_{\alpha}(q)
  =P_{\alpha}(0)=s_0=\sup\frac{-\sum_{i=1}^Nq_i\log q_i}{-\sum_{i=1}^Nq_i\log\lambda_i}.
  $$
  If $c<0$ then $\inf_{q>0}P_\alpha(q)=-\infty$, and $\sum_{i=1}^Nq_i\log
  p_i+\alpha=-c\sum_{i=1}^Nq_i\log \lambda_i<0$, thus, the claim of the lemma is trivial.

  If $P_\alpha''(q)>0$ for every $q\in\bbR$ then by \cref{thm:legendre}, there exists a unique
  $q^*\in\bbR$ such that $\inf_{q\in\bbR}P_\alpha(q)=P_\alpha(q^*)$. If $q^*\leq0$ then
  $\inf_{q>0}P_\alpha(q)=P_\alpha(0)$, and $\sum_{i=1}^N\lambda_i^{P_\alpha(0)}\log p_i+\alpha\geq0$
  by \cref{thm:legendre} and $\frac{-\sum_{i=1}^N\lambda_i^{P_\alpha(0)}\log
  \lambda_i^{P_\alpha(0)}}{-\sum_{i=1}^N\lambda_i^{P_\alpha(0)}\log\lambda_i}=P_\alpha(0)=s_0$. 
  If $q^*>0$ then $P_\alpha'(q^*)=0$ and by \cref{eq:Pder}
  \[
    0=\sum_{i=1}^N\lambda_i^{P_\alpha(q^*)}(p_ie^{\alpha})^{q^*}\log p_i+\alpha,
  \]
  and then by \cref{eq:corolary}
  \[
    P_\alpha(q^*)\geq\frac{-\sum_{i=1}^N\lambda_i^{P_\alpha(q^*)}(p_ie^{\alpha})^{q^*}\log
      \lambda_i^{P_\alpha(q^*)}(p_ie^{\alpha})^{q^*}}{-\sum_{i=1}^N\lambda_i^{P_\alpha(q^*)}
    (p_ie^{\alpha})^{q^*}\log\lambda_i}=P_{\alpha}(q^*)-q^*P_{\alpha}'(q^*)=P_{\alpha}(q^*).
  \]
  The proof of the second claim is similar and left for the reader.
\end{proof}

The three cases in \cref{thm:legendre} naturally define measures on $\Sigma$. Namely, in the first
case let us denote by $\bbQ_1$ the Bernoulli measure on $\Sigma$ associated to the probability vector
$\{\lambda_i^{P_\alpha(1)}p_ie^{\alpha}\}_{i=1}^N$, let us denote by $\bbQ_0$ the Bernoulli measure
on $\Sigma$ associated to the probability vector $\{\lambda_i^{s_0}\}_{i=1}^N$, and let $\bbQ_{q^*}$
be associated to
$\{q_{i,\alpha}^*=\lambda_i^{P_\alpha(q^*(\alpha))}(p_ie^{\alpha})^{q^*(\alpha)}\}_{i=1}^N$, where
$q^*(\alpha)$ is the
unique solution for $P_\alpha(q^*(\alpha))=\inf_qP_\alpha(q)$.

We will turn now to study the relation of \cref{eq:probpres} and $P_\alpha(q)$. Let
$\ell\colon\bbN\to\bbN$ be a monotone increasing function as in \cref{thm:main}. Let us write
\begin{equation}\label{eq:famousmn}
  m(n)=\left\lfloor \frac{\# \ell^{-1}(n)}{n} \right\rfloor.
\end{equation}
Note that $\limsup_{n\to\infty}\frac{\log m(n)}{n}=\alpha$.

For any $s\geq0$, we define the {\bf probabilistic pressure} as
\begin{equation}\label{eq:probpressure}
  L\colon s\mapsto\limsup_{n\to \infty} \tfrac 1n \log
  \sum_{|\bj|=n}\lambda_{\bj}^{s}\left(1-(1-p_{\bj})^{m(n)}\right).
\end{equation}
It is clear that the map above is continuous and strictly monotone decreasing. Moreover, it tends
to $-\infty$ as $s\to\infty$ and at $s=0$ it is non-negative. Indeed,
\[
  \begin{split}
    \sum_{|\bj|=n}\left(1-(1-p_{\bj})^{m(n)}\right)&\geq \sum_{|\bj|=n}p_{\bj}=1.
  \end{split}
\]
Hence, there is a unique solution of $L(s)=0$, which we denote by $s(\alpha)$ with a slight abuse of notation.

Next, we show that the two formulas of the dimension of the random covering set are
equivalent. Before doing so, we define two sets, which play an important role in this argument and further
in the paper. Let $\epsilon>0$ and
\begin{equation}\label{eq:GV}
  \cG_n^\epsilon=\{\bj\in \Sigma_n\mid 1-(1-p_\bj)^{m(n)}\le (1-\epsilon)^n \} \text{ and
  }\cV_n^\epsilon=\{\bj\in \Sigma_n \mid 1-(1-p_\bj)^{m(n)}>(1-\epsilon)^n \}.
\end{equation}
By the approximations \cref{eq:repeat} we immediately obtain
\begin{equation}\label{eq:setest}
  \{\bj\in \Sigma_n\mid p_{\bj}m(n)\le (1-\epsilon)^n\}\subset \cG_n^\epsilon \subset \{ \bj\in
  \Sigma_n\mid p_\bj m(n)\le \frac1\epsilon(1-\epsilon)^n \},
\end{equation}
and so,
\begin{equation}\label{eq:setest2}
  \{ \bj\in \Sigma_n\mid p_\bj m(n)> \frac1\epsilon(1-\epsilon)^n \}\subset \cV_n^\epsilon \subset
  \{\bj\in \Sigma_n\mid p_{\bj}m(n)> (1-\epsilon)^n\}.
\end{equation}

\begin{lemma}\label{thm:twospec}
  Let $s(\alpha)$ be the unique solution of $L(s(\alpha))=0$. Then $\inf_{q\in[0,1]}P_\alpha(q)=s(\alpha)$.
\end{lemma}

\begin{proof}
  Let us first suppose that $0<\alpha<-\sum_{i=1}^N\lambda_i^{P_\alpha(1)}p_ie^\alpha \log p_i$. By
  \cref{eq:repeat}, for every $s\geq0$
  \[\begin{split}
    L(s)=\limsup_{n\to \infty} \tfrac 1n \log
    \sum_{|\bj|=n}\lambda_{\bj}^{s}\left(1-(1-p_{\bj})^{m(n)}\right)
      &\leq\limsup_{n\to \infty}
      \tfrac 1n \log \sum_{|\bj|=n}\lambda_{\bj}^{s}m(n)p_{\bj}\\
      &=\log\sum_{i=1}^N\lambda_i^sp_i+\limsup_{n\to\infty}\tfrac1n\log
      m(n)\\
      &=\log\sum_{i=1}^N\lambda_i^sp_i+\alpha,
  \end{split}\]
  where in the last equality we have used the monotonicity of $\ell$ and \cref{eq:condonl}. Thus, in
  particular, $L(P_\alpha(1))\leq0$.

  Choose $\epsilon>0$ such that $\alpha+\sum_{i=1}^N\lambda_i^{P_\alpha(1)}p_ie^\alpha \log
  p_i<2\log(1-\epsilon)\leq0$. For sufficiently large $n\geq1$, $m(n)\leq e^{\alpha
  n}(1-\epsilon)^{-n}$. Then by \cref{eq:taylor} and the definition of $\cG_n^\epsilon$.
  \begin{align*}
    L(P_\alpha(1))
	&\geq\limsup_{n\to \infty} \tfrac 1n \log \sum_{\bj\in
	\cG_n^\epsilon}\lambda_{\bj}^{P_\alpha(1)}\left(1-(1-p_{\bj})^{m(n)}\right)\\
	&\geq\limsup_{n\to \infty} \tfrac 1n \log \sum_{\bj\in \cG_n^\epsilon}\lambda_{\bj}^{P_\alpha(1)}m(n)p_{\bj}\left(1-\frac{m(n)p_{\bj}}{2}\right)\\
	&\geq\limsup_{n\to \infty} \tfrac 1n \log \sum_{\bj\in \cG_n^\epsilon}\lambda_{\bj}^{P_\alpha(1)}m(n)p_{\bj}\left(1-\frac{(1-\epsilon)^n}{2\epsilon}\right)\\
	&\geq\limsup_{n\to \infty} \left(\tfrac 1n \log\left( \sum_{\bj\in \cG_n^\epsilon}\lambda_{\bj}^{P_\alpha(1)}e^{\alpha n}p_{\bj}\right)+\tfrac 1n \log (m(n)/e^{\alpha n})\right)\\
	&\geq\limsup_{n\to \infty} \Bigg(\tfrac 1n \log \bbQ_1\left(\left\{\bi\in\Sigma:\tfrac1n\log
	  p_{\bi|_n}+\frac{\log m(n)}{n}\leq\log(1-\epsilon)\right\}\right)\\
	&\hspace{10cm}+\tfrac 1n \log (m(n)/e^{\alpha n})\Bigg)\\
	&\geq\limsup_{n\to \infty} \left(\tfrac 1n \log \bbQ_1\left(\left\{\bi\in\Sigma:\tfrac1n\log p_{\bi|_n}+\alpha\leq2\log(1-\epsilon)\right\}\right)+\tfrac 1n \log (m(n)/e^{\alpha n})\right)
	\\
	&=0,
  \end{align*}
  where the last equality follows by the fact that the expected value of $\log p_{i_1}$
  w.r.t.~$\bbQ_1$ is $\sum_{i=1}^N\lambda_i^{P_\alpha(1)}p_ie^\alpha \log p_i$ and the choice of
  $\epsilon$.

  Now suppose that $-\sum_{i=1}^N\lambda_i^{P_\alpha(1)}p_ie^\alpha \log
  p_i\leq\alpha\leq-\sum_{i=1}^N\lambda_i^{P_\alpha(0)}\log p_i$. Let $\varepsilon>0$ be arbitrary.
  Then $m(n)\leq e^{\alpha n}(1-\epsilon)^{-n}$ for every sufficiently large $n\geq1$ and by
  \cref{thm:varprinc} and \cref{eq:setest2}
  \begin{align*}
    \mbox{}&L(P_\alpha(q^*(\alpha)))
    \\
	   &\leq\limsup_{n\to \infty} \tfrac 1n \log\left( 
	     \sum_{\bj\in \cG_n^\epsilon}\lambda_{\bj}^{P_\alpha(q^*(\alpha))}m(n)p_{\bj}+\sum_{\bj\in
	   \cV_n^\epsilon}\lambda_{\bj}^{P_\alpha(q^*(\alpha))}\right)\\
	   &\leq\limsup_{n\to \infty} \tfrac 1n \log\left( \sum_{\bj\in
	     \cG_n^\epsilon}\lambda_{\bj}^{P_\alpha(q^*(\alpha))}(m(n)p_{\bj})^{q^*(\alpha)}(1-\epsilon)^{n(1-q^*(\alpha))}+\sum_{\bj\in
	   \cV_n^\epsilon}\lambda_{\bj}^{P_\alpha(q^*(\alpha))}\right)\\
	   &\leq \limsup_{n\to \infty} \tfrac 1n
	   \log\left((1-\epsilon)^{n(1-2q^*(\alpha))}\bbQ_{q^*}(\{\bi:\bi|_n\in\cG_n^\epsilon\})\right.\\
	   &\left.\hspace{4cm}+ \sum_{\substack{k_1,\ldots,k_N\geq0\\k_1+\cdots+k_N=n\\n\alpha+\sum_{i}k_i\log
	   p_i>n2\log(1-\epsilon)}}\binom{n}{k_1,\ldots,k_N}\left(\lambda_1^{k_1}\cdots\lambda_N^{k_N}\right)^{P_{\alpha}(q^*(\alpha))}\right)\\
	   &=\limsup_{n\to \infty} \tfrac 1n
	   \log\left(e^{o(n)}(1-\epsilon)^{nq^*}\bbQ_{q^*}(\{\bi:\bi|_n\in\cG_n^\epsilon\})\right.\\
	   &\left.\hspace{4cm}+
	   \binom{n+N-1}{N}e^{n\left(h_{\bbQ_{q^*(\alpha-2\log(1-\epsilon))}}+P_{\alpha}(q^*(\alpha))\sum_{i}q_{i,\alpha-2\log(1-\epsilon)}^*\log\lambda_i\right)}\right)\\
	   &\leq\max\left\{ (1-q^*(\alpha))\log(1-\epsilon),\left(P_\alpha(q^*(\alpha))-P_{\alpha-2\log(1-\epsilon)}(q^*(\alpha-2\log(1-\epsilon)))\right)\log\lambda_{\min}\right\}.
  \end{align*}
  and since $\epsilon>0$ was arbitrary, we get that $L(P_\alpha(q^*(\alpha)))\leq0$.

  For the other inequality, observe that for every $\epsilon>0$ there exists a subsequence $n_k$ such that
  \begin{equation}\label{eq:this}
    \frac{(1-\epsilon)^{n_k}}{\epsilon}<p_{\min}e^{-\alpha
    n_k}m(n_k)<p_1^{q_{1,\alpha}^*n_k}\cdots p_N^{q_{N,\alpha}^*n_k}m(n_k),
  \end{equation}
  and so, every finite word $\bi$ with $|\bi|=n_k$ and $\#_i\bi=q_{i,\alpha}^*n_k$ belongs to
  $\cV_{n_k}^\epsilon$. Hence,
  \[
    \begin{split}
      L(P_\alpha(q^*(\alpha)))
	&\geq\limsup_{n\to \infty} \tfrac 1n \log\left( \sum_{\bj\in
	\cV_n^\epsilon}\lambda_{\bj}^{P_\alpha(q^*(\alpha))}(1-\epsilon)^n\right)\\
	&\geq\limsup_{k\to \infty} \tfrac{1}{n_k} \log\left(
	\binom{n_k}{q_{1,\alpha}^*n_k,\ldots,q_{N,\alpha}^*n_k}\left(\lambda_{1}^{q_{1,\alpha}^*n_k}\cdots\lambda_{N}^{q_{N,\alpha}^*n_k}\right)^{P_\alpha(q^*(\alpha))}(1-\epsilon)^{n_k}\right)\\
	&\geq\log(1-\epsilon),
    \end{split}
  \]
  and since $\epsilon>0$ was arbitrary $L(P_\alpha(q^*(\alpha)))\geq0$.

  Finally, suppose that $\alpha>-\sum_{i=1}^N\lambda_i^{P_\alpha(0)}\log p_i$. Since
  $L(s)\leq\log\sum_{i=1}^N\lambda_i^s$, we can further conclude that
  $L(P_\alpha(0))\leq0$. On the other hand, similarly to \cref{eq:this} one can show that
  $(1-\epsilon)^{n_k}/\epsilon<p_1^{\lambda_1^{s_0}n_k}\cdots p_N^{\lambda_1^{s_0}n_k}m(n_k)$ over a
  subsequence $n_k$, and so, similar argument leads us to the inequality $L(P_\alpha(0))\geq0$.
\end{proof}

Finally, let us note that a simple corollary of \cref{thm:twospec} is that the probabilistic
pressures defined in \cref{eq:probpres} and \cref{eq:probpressure} have the same root.

\section{Upper bounds}\label{sec:upper}

This section is devoted to proving the almost sure upper bounds in the main theorems.

\subsection{Upper bound for the dynamical covering set}

First, we will prove the upper bound to the Hausdorff dimension in \cref{thm:main}. Notice that
the value $s(\alpha)$ given as the root of the probabilistic pressure from
\cref{eq:probpressure}, by \cref{thm:twospec} is equal to the dimension value $\inf_{q\in
[0,1]}P_{\alpha}(q)$ from Theorem \cref{thm:main}. 

\begin{proposition}\label{thm:ubmain1}
  Let $\{f_i(x)=\lambda_iO_i+t_i\}_{i=1}^N$ be an IFS of similarities. Let $\alpha>0$ and assume that $\ell: \mathbb
  N\to \mathbb N$ is monotone increasing and it satisfies
  \begin{equation*}
    \liminf_{n\to \infty} \frac{\ell(n)}{\log n}=\tfrac 1\alpha.
  \end{equation*}
  Let $\bbP_p$ be a Bernoulli measure corresponding to a probability vector $p=(p_1, \dots, p_N)$.
  Then, for $\bbP_{p}$-almost every $\bo \in \Sigma$
  $$
  \dim_HR(\bo,\ell)\leq s(\alpha),
  $$
  where $s(\alpha)$ is the unique root of the probabilistic pressure defined in \cref{eq:probpressure}.
\end{proposition}

\begin{proof}
  Notice that, for all
  $m\in \bbN$
  \begin{equation*}
    R(\bo, \ell)=\limsup_{n\to\infty}\pi[\sigma^n(\bo)|_{\ell_n}]\subset \bigcup_{n\ge
    m}\bigcup_{\substack{|\bi|=n\\ \sigma^k(\bo)|_{\ell_k}=\bi\text{ for some $k$}}}\pi[\bi],
  \end{equation*}
  and (up to an absolute constant) $\diam(\pi[\bi])=\lambda_{\bi}$. Hence, the Hausdorff dimension of $R(\bo,
  \ell)$ is bounded from above by all those $t$ for which
  \[
    \sum_{n=1}^\infty \sum_{|\bi|=n}\lambda_{\bi}^t\cdot\mathbb 1\{\bi=\sigma^k(\bo)|_n \text{ for
    some $k$ for which }n=\ell(k)\}<\infty.
  \]
  It is difficult to bound this series for a fixed $\bo$, but as we are only looking for bounds that
  hold for $\bbP_p$-almost every $\bo$, we aim to prove, instead, that
  \begin{equation}\label{eq:expectationest}
    \begin{split}
&\bbE\left(\sum_{n=1}^\infty \sum_{|\bi|=n}\lambda_{\bi}^t\cdot\mathbb 1\{\bi=\sigma^k(\bo)|_n
\text{ for some $k$ for which }n=\ell(k)\}\right)\\
&\qquad=\sum_{n=1}^\infty \sum_{|\bi|=n}\lambda_\bi^t\bbP_p (\bi=\sigma^k(\bo)|_n \text{ for some $k$
for which }n=\ell(k))<\infty.
    \end{split}
  \end{equation}
  Here the expectation is with respect to the Bernoulli measure $\bbP_p$. For any fixed $n$ there are
  $\#\ell^{-1}(n)$ values of $k$ for which $n=\ell(k)$. Among these options, there are
  \begin{equation}\label{eq:defmn}
    \left\lfloor\frac{\#\ell^{-1}(n)}{n}\right\rfloor=m(n)
  \end{equation}
  such that correspond to disjoint stretches of $\bo$. This is due to the fact that, denoting by
  $\min\ell^{-1}(n)=:M$ for any fixed $n$, there are approximately $m(n)$ elements $p$ in the
  collection
  \[
    P_k(n)=\{M+k, M+k+n, \dots, M+k+qn\mid q\in \mathbb N\text{ maximal s.t. } M+k+qn\in \ell^{-1}(n)\},
  \]
  where $k\in\{0,\ldots,n-1\}$. These correspond to disjoint stretches $\sigma^p(\bo)$. Hence we
  have, for any fixed $n$, $\bi$
  \begin{align}
    \bbP (\bi=\sigma^p(\bo)|_n \text{ for some $p$ for which }n=\ell(p))
  &\le \sum_{k=0}^{n-1}\bbP(\bi=\sigma^p(\bo)|_n \text{ for some }p\in P_k(n))\nonumber\\
  &\le n\left(1-(1-p_{\bi})^{m(n)}\right).\label{eq:standard}
  \end{align}
  Hence, when $t> s(\alpha)$, as $L$ is strictly decreasing,
  \[
    \limsup _{n\to \infty} \tfrac 1n \log \sum_{|\bj|=n}\lambda_{\bj}^{t}(1-(1-p_{\bj})^{m(n)})<0,
  \]
  then
  \[
    \sum_{n=1}^\infty n\sum_{|\bi|=n}\lambda_\bi^t(1-(1-p_{\bi})^{m(n)})<\infty,
  \]
  and in particular \eqref{eq:expectationest} and hence the upper bound to the Hausdorff dimension of
  $R(\bo, \ell)$ holds.
\end{proof}

\subsection{Upper bound for the complement}\label{sec:upperforcomp}

Now, we verify the upper bound of \cref{thm:main2}.

\begin{proposition}\label{thm:ubmain2}
  Let $\{f_i(x)=\lambda_iO_i+t_i\}_{i=1}^N$ be an IFS of similarities. Let $\alpha>0$ and assume that $\ell: \mathbb
  N\to \mathbb N$ is monotone increasing and it satisfies
  \begin{equation*}
    \liminf_{n\to \infty} \frac{\ell(n)}{\log n}=\tfrac 1\alpha.
  \end{equation*}
  Let $\bbP_p$ be a Bernoulli measure corresponding to a probability vector $p=(p_1, \dots, p_N)$.
  Then, for $\bbP_{p}$-almost every $\bo \in \Sigma$
  $$
  \dim_HR(\bo,\ell)^c\leq t(\alpha),
  $$
  where $t(\alpha)$ is defined in \cref{eq:salpha2}.
\end{proposition}

\begin{proof}
  Notice that,
  \[
    R(\bo, \ell)^c= \bigcup_{m}\bigcap_{n\ge m}\bigcup_{\substack{|\bi|=n\\
    \sigma^k(\bo)|_{\ell_k}\neq\bi\text{ for every $k$}}}\pi[\bi],
  \]
  Let $\epsilon>0$ be arbitrary but fixed. Let $(n_p)$ be the subsequence for which
  $$
  m(n_p)e^{\alpha n_p}\geq(1+\epsilon)^{-n_p},
  $$
  where $m(n)$ is as in \cref{eq:defmn}. So, for every $m\geq1$
  \[
    R(\bo, \ell)^c\subset \bigcup_{p=m}^\infty\bigcup_{\substack{|\bi|=n_p\\
    \sigma^k(\bo)|_{\ell_k}\neq\bi\text{ for every $k$}}}\pi[\bi].
  \]
  So, similarly to the previous case, the Hausdorff dimension of $R(\bo, \ell)^c$ is bounded from above by $t$ if
  $$
  \sum_{p=1}^\infty\sum_{|\bi|=n_p}\lambda_{\bi}^t\cdot\mathbb 1\{\bi\neq\sigma^k(\bo)|_{n_p} \text{
  for every $k$ for which }n_p=\ell(k)\}<\infty.
  $$
  To show that $\dim_HR(\bo,\ell)^c\leq t$ for $\bbP_p$-almost every $\bo$, it is enough to show that
  \begin{multline*}
    \bbE\left(\sum_{p=1}^\infty\sum_{|\bi|=n_p}\lambda_{\bi}^t\cdot\mathbb
      1\{\bi\neq\sigma^k(\bo)|_{n_p} \text{ for every $k$ for which
    }n_p=\ell(k)\}\right) \\
    \leq\sum_{p=1}^\infty\sum_{|\bi|=n_p}\lambda_{\bi}^t\cdot(1-p_{\bi})^{m(n_p)}<\infty.
  \end{multline*}
  For $\epsilon>0$, let us now define
  \begin{equation}\label{eq:modifGV}
    \mathcal{G'}_n^\epsilon=\{\bj\in\Sigma_n:p_{\bi}m(n)\geq(1+\epsilon)^n\}\quad\text{and}\quad
    \mathcal{V'}_n^\epsilon=\{\bj\in\Sigma_n:p_{\bi}m(n)<(1+\epsilon)^n\}.
  \end{equation}
  Then by choosing
  $t=(1+\epsilon)P_{\alpha+2\log(1+\epsilon)}(q^*(\alpha)+2\log(1+\epsilon))=(1+\epsilon)\inf_{q<0}P_{\alpha+2\log(1+\epsilon)}(q)$
  and \cref{eq:repeat}
  \[\begin{split}
    \sum_{|\bi|=n_p}\lambda_{\bi}^t\cdot(1-p_{\bi})^{m(n_p)}
  &\leq\sum_{\bi\in\mathcal{G'}_{n_p}^\epsilon}\lambda_{\bi}^te^{-p_{\bi}m(n_p)}+\sum_{\bi\in\mathcal{V'}_{n_p}^\epsilon}\lambda_{\bi}^t\\
  &\leq e^{-(1+\epsilon)^{n_p}}N^{n_p}+\sum_{\substack{k_1,\ldots,k_N\geq0\\ k_1+\cdots+k_N=n_p \\ n_p\alpha+\sum_{i}k_i\log p_i<2n_p\log(1+\epsilon)}}\binom{n}{k_1,\ldots,k_N}\left(\lambda_1^{k_1}\cdots\lambda_N^{k_N}\right)^t\\
  &\leq e^{-(1+\epsilon)^{n_p}}N^{n_p}+\binom{n_p+N-1}{n_p}e^{n_p\left(h_{\bbQ_{q^*(\alpha+2\log(1+\epsilon))}+t\sum_{i}q_{i,\alpha+2\log(1+\epsilon)}^*\log\lambda_i}\right)}\\
  &\leq e^{-(1+\epsilon)^{n_p}}N^{n_p}+\binom{n_p+N-1}{n_p}\lambda_{\max}^{n_p\epsilon P_{\alpha+2\log(1+\epsilon)}(q^*(\alpha+2\log(1+\epsilon))}.
  \end{split}\]
  The right-hand side forms a convergent series, and since $\epsilon>0$ was arbitrary and the map
  $\alpha\to\inf_{q}P_\alpha(q)$ is continuous, the upper bound follows.
\end{proof}

\section{Lower bounds for ``large'' \texorpdfstring{$\alpha$}{a}}

\label{sec:lowerlarge}

\subsection{The region \texorpdfstring{$\alpha>\alpha_2$}{large a}}

We begin this section with a proof that for high $\alpha$, the dynamical covering set turns out to
be the entire attractor $\Lambda$. This is captured by the
following theorem. Recall the notation $\alpha_2=-\log p_{\min}$. 
\begin{theorem}\label{thm:fullCover}
  Let $\{f_i(x)=\lambda_iO_i+t_i\}_{i=1}^N$ be an IFS of similarities with attractor $\Lambda$ and
  let $\bbP_p$ be a Bernoulli measure corresponding to a probability vector $p=(p_1, \dots, p_N)$.
  Let $\alpha>0$ be such that $\alpha>\alpha_2$ and assume that $\ell: \mathbb
  N\to \mathbb N$ is monotone increasing and it satisfies
  \begin{equation*}
    \liminf_{n\to \infty} \frac{\ell(n)}{\log n}=\tfrac 1\alpha.
  \end{equation*}
  Then, for
  $\bbP_p$-almost all $\bo\in\Sigma$,
  \[
    R(\bo, \ell) =  \Lambda.
  \]
\end{theorem}

\begin{proof}
  Let $\epsilon>0$ be such that $p_{\min}e^{\alpha}(1-\epsilon)>1$. Let $m(n)$ be as in
  \cref{eq:famousmn}, and let $\{n_p\}$ be the subsequence such that $\limsup_{n\to
  \infty}\frac{m(n)}{n}=\lim_{p\to\infty}\frac{m(n_p)}{n_p}=\alpha$. Without loss of generality we
  may assume that $m(n_p)\geq e^{\alpha n_p}(1-\epsilon)^{n_p}$.

  We use a Borel-Cantelli argument to show that for typical $\bo$, for sufficiently large $p$ and
  for every cylinder $[\bj]$, where
  $\bj\in\Sigma_{n_p}$, there exists a $k$ such that $\ell(k) = n_p$ and
  $\sigma^k(\bo)|_{\ell(k)}=\bj$. In particular, we show that the set
  $$
  \bigcap_{k=1}^\infty\bigcup_{p=k}^\infty\bigcup_{\bj\in\Sigma_{n_p}}\{\bo:(\sigma^q\bo)|_{n_p}\neq\bj\text{
  for every $q\in\bbN$ with $\ell(q)=n_p$}\}
  $$
  has zero $\bbP$ measure. This then shows that $R(\bo, \ell) = \Lambda$.

  First write $M_p$ for the least integer and $M_p'$ for the last integer such that
  $\ell(M_p)=\ell(M_p')=n_p$.  We want to estimate the probability of the event that there exists
  at least one cylinder $\bj\in\Sigma_{n_p}$ that does not appear in $\bo$ between $M_p$ and
  $M_p'$. Denote this probability by $P_p$. Clearly, it is bounded by
  \begin{align*}
    P_p
    &\leq \sum_{\bj\in\Sigma_{n_p}}\bbP\{\bo\in\Sigma : (\sigma^k\bo)|_{n_p}\neq \bj, M_p\leq k <
    M_p'\}\\
    &\leq \sum_{\bj\in\Sigma_{n_p}}\bbP\left\{\bo\in\Sigma : \sigma^{M_p+n_pk}\bo|_{n_p}\neq \bj, 0\leq k
    <\frac{M_{p}'-M_p}{n_p}\right\}\\
    &=\sum_{\bj\in\Sigma_{n_p}}\left( \prod_{k=0}^{(M_p'-M_p)/n_p-1} \left(1-\bbP\left\{\bo\in\Sigma :
    \sigma^{M_p+n_pk}\bo|_{n_p}= \bj\right\} \right)\right)\\
    &=\sum_{\bj\in\Sigma_{n_p}}\left( 1-\bbP([\bj]) \right)^{(M_p'-M_p)/n_p}\\
    &\leq N^{n_p} \left( 1-p_{\min}^{n_p} \right)^{m(n_p)/n_p}
    =N^{n_p} \exp\log\left( 1-p_{\min}^{n_p} \right)^{m(n_p)/n_p}
    \intertext{and since $\log(1-x)\leq -x$ for all $x<1$,}
    &\leq N^{n_p} \exp \left(-\frac{m(n_p)p_{\min}^{n_p}}{n_p}\right)\leq N^{n_p} \exp
    \left(-\frac{e^{\alpha n_p}(1-\epsilon)^{n_p}p_{\min}^{n_p}}{n_p}\right).
  \end{align*}
  By assumption, $e^\alpha p_{\min}(1-\epsilon)>0$, and so, this bound decreases
  super-exponentially and $\sum_p P_p < \infty$. By the Borel-Cantelli
  lemmas we see that for almost every $\bo$ there exists a level $K_{\bo}$ such that for every cylinder $[\bj]$ where
  $\bj\in\Sigma_{n_p}$, $p>K_{\bo}$ there exists a $k$ such that $\ell(k) = n_p$ and
  $\sigma^k(\bo)|_{\ell(k)}=\bj$. It follows that $R(\bo, \ell)=\Lambda$, as required.
\end{proof}

\subsection{The region \texorpdfstring{$\alpha_1<\alpha<\alpha_2$}{intermediate probabilities}}\label{sec:lbver1}

In this region, we will study the Hausdorff dimension of $R(\bo,\ell)^c$ via the Bernoulli measures
of the set $R(\bo,\ell)$ with respect to measures for which $\sum_iq_i\log p_i+\alpha<0$. Recall that 
\[
  \alpha_1=- \sum_{i=1}^N \lambda_i^{s_0}\log p_i, 
\]
where $s_0=\dim_H\Lambda$. Also recall $\alpha_2=-\log p_{\min}$.  

\begin{theorem}\label{thm:berR0}
  Let $\alpha_1<\alpha$ and further assume that $\ell(n)$ satisfies
  \[
    \liminf_{n\to\infty}\frac{\ell(n)}{\log n} = \frac{1}{\alpha}.
  \]
  Let $\nu$ be a Bernoulli measure with probability
  vector $q=(q_1,\ldots,q_N)$ such that $\sum_iq_i\log p_i+\alpha<0$. Then, $\nu(R(\bo,\ell)) =0$ for
  $\bbP$-almost every $\bo$.
\end{theorem}

First, we prove the corollary of this result.

\begin{corollary}\label{thm:notfullregion}
  Let $\alpha_1<\alpha<\alpha_2$, and assume
  that $\ell(n)$ satisfies
  $\liminf_{n\to\infty}\ell(n)/\log n = 1/\alpha$. Then, $\dim_H(R(\bo,\ell)^c)
  =\inf_{q<0}P_\alpha(q)<s_0$ and $\cH^{s_0}(R(\bo,\ell))=\cH^{s_0}(\Lambda)$, where $s_0 = \dimh
  \Lambda$ for $\bbP$-almost every $\bo$.
\end{corollary}

\begin{proof}
  For $\alpha$ in this range, by \cref{thm:varprinc} and \cref{thm:ubmain2}, 
  \[\dim_H(R(\bo,\ell)^c) \leq\inf_{q<0}P_\alpha(q)<s_0\]
  and so, the
  claim for the Hausdorff measure of $R(\bo,\ell)$ follows.

  Let $q=(q_1,\ldots,q_N)$ be the probability vector, where the supremum in \cref{eq:varprinc2}
  is attained. Then by \cref{thm:varprinc}, $\sum_iq_i\log p_i+\alpha<0$, and so, by
  \cref{thm:berR0},
  $$
  \dim_H(R(\bo,\ell)^c)\geq\sup\left\{\frac{-\sum_{i=1}^Nq_i\log q_i}{-\sum_{i=1}^Nq_i\log\lambda_i}
  : \sum_{i=1}^Nq_i\log p_i+\alpha>0\right\}=\inf_{q>0}P_\alpha(q).
  $$
\end{proof}

\begin{proof}[Proof of \cref{thm:berR0}]
  The proof is similar to \cref{sec:upper}. That is, by Borel-Cantelli's lemma, it is enough to show
  that for $\bbP$-almost every $\bo$
  \[
    \sum_{n=1}^\infty \sum_{|\bi|=n}q_{\bi}\cdot\mathbb 1\{\bi=\sigma^k(\bo)|_n \text{ for some $k$
    for which }n=\ell(k)\}<\infty.
  \]
  Similarly to \cref{eq:expectationest} and \cref{eq:standard}, it is enough to show that
  \[
    \sum_{n=1}^\infty \sum_{|\bi|=n}q_{\bi}\cdot n(1-(1-p_{\bi}))^{m(n)}<\infty.
  \]
  Let $\epsilon>0$ be such that $\sum_{i=1}^Nq_i\log p_i+\alpha<2\log(1-\epsilon)<0$. Recalling the
  definition of $\mathcal{G}_n^\epsilon$ and $\mathcal{V}_n^\epsilon$ from \cref{eq:GV} together
  with \cref{eq:setest} and \cref{eq:setest2}, we get for sufficiently large $n$ that
  \[\begin{split}
    \sum_{|\bi|=n}q_{\bi}\cdot n(1-(1-p_{\bi}))^{m(n)}
    &\leq \sum_{\bi\in\mathcal{G}_n^\epsilon}q_{\bi}m(n)p_{\bi}+\sum_{\bi\in\mathcal{V}_n^\epsilon}q_{\bi}\\
    &\leq \tfrac1\epsilon(1-\epsilon)^n+\nu\left(\left\{\bi: p_{\bi|_n}m(n)>(1-\epsilon)^n\right\}\right)
    \intertext{by $m(n)\leq(1-\epsilon)^{-n}e^{\alpha n}$ for $n$ sufficiently large}
    &\leq \tfrac1\epsilon(1-\epsilon)^n+\nu\left(\left\{\bi: p_{\bi|_n}e^{n\alpha}>(1-\epsilon)^{2n}\right\}\right)\\
    &\leq \tfrac1\epsilon(1-\epsilon)^n+e^{-\delta n},
  \end{split}\]
  where the last inequality follows by the large deviation principle for some
  $\delta=\delta(q,\epsilon)$. Then $\sum_{n=1}^\infty
  \sum_{|\bi|=n}q_{\bi}\cdot n(1-(1-p_{\bi}))^{m(n)}<\infty$ clearly follows.
\end{proof}

\subsection{The region \texorpdfstring{$\alpha_0<\alpha<\alpha_1$}{low values}}

Similarly to \cref{sec:lbver1}, we will study the Hausdorff dimension of $R(\bo,\ell)$ via the
Bernoulli measures of the set $R(\bo,\ell)^c$ with respect to measures for which $\sum_iq_i\log
p_i+\alpha>0$.

\begin{theorem}\label{thm:berRc0}
  Let $\alpha>0$ be such that $\alpha<\alpha_1$ and further assume that $\ell(n)$ satisfies
  $\liminf_{n\to\infty}\ell(n)/\log n = 1/\alpha$. Let $\nu$ be a Bernoulli measure with probability
  vector $q=(q_1,\ldots,q_N)$ such that $\sum_iq_i\log p_i+\alpha>0$. Then, $\nu(R(\bo,\ell)^c) =0$ for
  $\bbP$-almost every $\bo$.
\end{theorem}

This result has the immediate corollary

\begin{corollary}\label{thm:spectrumregion}
  Let  $\alpha_0<\alpha<\alpha_1$ and assume that $\ell(n)$ satisfies
  $\liminf_{n\to\infty}\ell(n)/\log n = 1/\alpha$. Then, $\dim_H(R(\bo,\ell))
  =\inf_{q\in[0,1]}P_\alpha(q)$ and also $\cH^{s_0}(R(\bo,\ell)^c)=\cH^{s_0}(\Lambda)$, where $s_0 =
  \dimh \Lambda$ for $\bbP$-almost every $\bo$.
\end{corollary}

\begin{proof}
  By \cref{sec:upper} and \cref{thm:twospec}, $\dim_H(R(\bo,\ell))
  \leq\inf_{q\in[0,1]}P_\alpha(q)<s_0$ and so, the claim for the Hausdorff measure of
  $R(\bo,\ell)^c$ follows.

  Let $q=(q_1,\ldots,q_N)$ be the probability vector, where the supremum in \cref{eq:varprinc2} is
  attained. Then by \cref{thm:varprinc}, $\sum_iq_i\log p_i+\alpha>0$. Thus, by \cref{thm:berRc0},
  $$
  \dim_H(R(\bo,\ell))\geq\sup\left\{\frac{-\sum_{i=1}^Nq_i\log q_i}{-\sum_{i=1}^Nq_i\log\lambda_i} :
  \sum_{i=1}^Nq_i\log p_i+\alpha>0\right\}=\inf_{q>0}P_\alpha(q),
  $$
  where in the last equality, we have used again \cref{thm:varprinc}.
\end{proof}

\begin{proof}[Proof of \cref{thm:berRc0}]
  The proof will be similar to \cref{sec:upperforcomp}.  Let $\epsilon>0$ and let $(n_p)$ be the
  subsequence for which
  $$
  m(n_p)e^{-\alpha n_p}\geq(1+\epsilon)^{-n_p}.
  $$
  By Borel-Cantelli's lemma, it is enough to show that for $\bbP$-almost every $\bo$
  $$
  \sum_{p=1}^\infty\sum_{|\bi|=n_p}q_{\bi}\cdot\mathbb 1\{\bi\neq\sigma^k(\bo)|_{n_p} \text{ for
  every $k$ for which }n_p=\ell(k)\}<\infty.
  $$
  To show this for $\bbP_p$-almost every $\bo$, it is enough to show that
  \begin{align}
    \bbE\left(\sum_{p=1}^\infty\sum_{|\bi|=n_p}q_{\bi}\cdot\mathbb 1\{\bi\neq\sigma^k(\bo)|_{n_p}
	\text{ for every $k$ for which
    }n_p=\ell(k)\}\right)\nonumber
    \\
    \leq\sum_{p=1}^\infty\sum_{|\bi|=n_p}q_{\bi}\cdot(1-p_{\bi})^{m(n_p)}<\infty.\label{eq:conv}
  \end{align}
  Choose $\epsilon>0$ such that $\sum_iq_i\log p_i+\alpha>2\log(1+\epsilon)>0$, and by recall the
  definition of $\mathcal{G'}_n^\epsilon$ and $\mathcal{V'}_n^\epsilon$ from \cref{eq:modifGV}, we
  get
  \[\begin{split}
    \sum_{|\bi|=n_p}q_{\bi}\cdot(1-p_{\bi})^{m(n_p)}
    &\leq\sum_{\bi\in\mathcal{G'}_{n_p}^\epsilon}q_{\bi}e^{-p_{\bi}m(n_p)}+\sum_{\bi\in\mathcal{V'}_{n_p}^\epsilon}q_{\bi}\\
    &\leq e^{-(1+\epsilon)^{n_p}}+\nu\left(\left\{\bi\in\Sigma: \alpha+\tfrac1n_p\log
    p_{\bi}<2\log(1+\epsilon)\right\}\right)\\
    &\leq e^{-(1+\epsilon)^{n_p}}+e^{-n_p\delta}
  \end{split}\]
  where the last inequality follows by the large deviation principle. Then \cref{eq:conv} clearly follows.
\end{proof}

\section{Lower bound for \texorpdfstring{$\alpha\le\alpha_0$}{small a}}
\label{sec:lowersmall}

At its heart, the strategy of proof of the lower bound from \cref{thm:main} in this case is standard: We
construct a Cantor subset to $R(\bo, \ell)$ in \cref{sec:cantor}, and then a measure on this
Cantor set in \cref{sec:randommeasure}. We prove well-distribution properties of the measure in
\cref{subsec:measure}, and show that $\dim_HR(\bo,\ell)\geq s(\alpha)$ via an energy estimate in
\cref{subsec:energy}. By far the most difficult and technical part of this proof is the definition of the
measure and handling its intricacies in order to compute its energy.

Let $\alpha_0\ge \alpha>0$ and $\ell\colon\bbN\to\bbN$ be an increasing function such that
$$
\liminf_{n\to\infty}\frac{\ell(n)}{\log n}=\frac1\alpha.
$$
In this section, our standing assumption is that
$\alpha\leq-\sum_{i=1}^N\lambda_i^{P_\alpha(1)}p_ie^{\alpha}\log p_i$. Let $\gamma<\alpha$ be
arbitrary. Then $\gamma<-\sum_{i=1}^N\lambda_i^{P_\gamma(1)}p_ie^{\gamma}\log p_i$ holds.

For short, we use the notation $s(\gamma)=P_\gamma(1)$ (recall \cref{thm:legendre} and
\cref{thm:twospec}). So, let us fix $\gamma\in\bbR$ such that
$\gamma<-\sum_{i=1}^N\lambda_i^{s(\gamma)}p_ie^{\gamma}\log p_i$. Then there exists
$\epsilon_\gamma>0$ such that $\gamma+\sum_{i=1}^N\lambda_i^{s(\gamma)}p_ie^{\gamma}\log
p_i<\log(1-\epsilon_\gamma)$ holds. Then by Cram\'er's bound, there exists $\beta=\beta_\gamma>0$
such that for every $n\in\bbN$
\begin{align}
  \bbQ_1\left(\left\{\bj\in\Sigma_{n}: p_{\bj}e^{\gamma |\bj|} > (1-\epsilon_\gamma)^{|\bj|}
  \right\}\right)\leq e^{-\beta n},\label{eq:Cramer}
\end{align}
where we recall that $\bbQ_1$ is the Bernoulli measure induced by the probabilities
$(\lambda_i^{s(\gamma)}p_ie^{\gamma})_{i=1}^N$.

\subsection{Cantor subset for the lower bound}\label{sec:cantor}

Denote, for each $n\in \bbN$, $M_n:=\min\ell^{-1}(n)$. Let $q$ be the maximal integer such that
$q\leq 2e^{\gamma n}/n$ and $M_n+qn\in \ell^{-1}(n)$. Let
\[
  P(n)=\{M_n, M_n+n, \dots, M_n +(q-1)n\}.
\]
Let $\{n_p\}$ be a sequence such that $\limsup_{n\to\infty}\frac{\log
m(n)}{n}=\lim_{p\to\infty}\frac{\log m(n_p)}{n_p}=\alpha>\gamma$. Without loss of generality, we may assume that
\begin{equation}\label{eq:mi}
  \frac{e^{\gamma n_p}}{2n_p}\leq\# P(n_p)\leq\frac{2e^{\gamma n_p}}{n_p}\quad\text{ for every }p\in\bbN.
\end{equation}

Let us now define a class of functions $n: \Sigma_*\to \bbN$, which plays a central role in our construction.

\begin{definition}\label{def:constructor}
  We call a function $n: \Sigma_*\to \bbN$ a \emph{constructor function} if it satisfies the
  following properties: For every $\bi, \bj\in \Sigma_*$
  \begin{enumerate}[({C}1)]
    \item
      $n(\bi)\in\{n_p\}$,
    \item if $\bi\neq \bj$, then $n(\bi)\neq n(\bj)$,
    \item\label{it:condni3} if $|\bj|>|\bi|$, then $n(\bj) > n(\bi)$
    \item\label{it:condni4} $\frac{n(\bi)}{p_\bi e^{\gamma|\bi|}}>2$,
    \item\label{it:condni5} $\frac{|\bi|}{n(\bi)}\leq\frac{1}{n(\emptyset)}\leq\frac12$.
  \end{enumerate}
\end{definition}
It is easy to construct constructor functions. Arrange the elements of $\Sigma_*$ such that
their length are increasing and the same length words are in alphabetic order. Let this be
$(\emptyset,1,\ldots,m,11,\ldots)=(\bi_0,\bi_1,\ldots)$ Fix $2\leq n(\emptyset)\in\{n_p\}$ being
arbitrary and then define $n(\bi_n)$ by induction such that let
$n(\bi_n)\in\{n_p\}\cap(n(\bi_{n-1}),\infty)$ be the smallest element such that
\cref{it:condni4} and \cref{it:condni5} are satisfied.

Denote, for $\bi\in \Sigma_*$
\[
  m(\bi):=\# P(n(\bi)).
\]
Set inductively, for $\bi\in\Sigma_*, n\in \bbN$,
\[
  \Xi_0^{\bi}=\{\bi\} \quad\text{and}\quad
  \Xi_n^\bi=\{\bj\in \Sigma_*\mid n(\bi')=|\bj|\text{ and
  }\bj|_{|\bi'|}=\bi'\text{ for some }\bi'\in \Xi_{n-1}^\bi\},
\]
the set of children of $\bi$ at the $n$th construction level. Observe that
$[\bi]=\bigcup\{[\bj]:\bj\in\Xi_n^\bi\}$ for every $n\geq0$, hence $\{\Xi_n^\bi\}_{n\in\bbN}$ form a
very sparse subtree of $\Sigma_*$, which still covers the whole set $\Sigma$. Now, we will choose
the random subset which are actually visited by the random cover process.

Let $\cC_1$ be the collection of finite words, which have been visited by the ``children of the empty
set''. That is,
$$
\cC_1:=\{(\sigma^p\bo)|_{n(\emptyset)} :p\in P(n(\emptyset))\}\subseteq\Xi_1^\emptyset
$$
Now, let us define the $n$th generation by induction, namely,
$$
\cC_n:=\bigcup_{\bi\in\cC_{n-1}}\{(\sigma^p\bo)|_{n(\bi)} :p\in P(n(\bi))\text{ }\&\text{
}\bi<\sigma^p\bo\}\subseteq\Xi_n^\emptyset.
$$
Observe that
$$
\{(\sigma^p\bo)|_{n(\bi)} :p\in P(n(\bi))\text{ }\&\text{ }\bi<\sigma^p\bo\}\subseteq\Xi_1^\bi.
$$
Then the random Cantor subset $C$ of $R(\bo,\ell)$ we consider is defined as intersection of the
nested sequence of compact sets $C_n\subset\Sigma$, where
\[
  C_{n}=\bigcup _{\bi\in\cC_n}[\bi]\qquad\text{and}\qquad C=\bigcap_{n=1}^\infty C_n.
\]
Let us observe that the random Cantor subset $C$ depends on the constructor function defined in
\cref{def:constructor}. Clearly, $C=\emptyset$ with positive probability. In particular, we will
construct a random measure $\mu=\mu_{\bo}$ for $\bbP$-almost every $\bo$ such that
$0<\mu(C)<\infty$ with positive probability. Our goal is to show that for every $\epsilon>0$ there
is a properly chosen constructor function such that the $(s(\gamma)-\epsilon)$-energy of $\mu$ is
finite for every $\epsilon>0$ almost surely. In the remaining part of the section, we suppress the
notation $\mu_\bo$, but the reader shall keep in mind that the measure actually depends on
$\bo$.

\subsection{Construction of a random measure on the Cantor set}\label{sec:randommeasure}

For every $\bi\in \Sigma_*$ we will define an $L^2$-martingale that will be used to determine the
measure of $[\bi]$. First we will need some notation. Let $s_\bi$ be such that
\[
  \sum_{|\bj|=n(\bi),\ \bi<\bj}\lambda^{s_\bi}_{\sigma^{|\bi|}(\bj)}(1-(1-p_\bj)^{m(\bi)})=1.
\]
Then recall the definition of $P(n)$ from \cref{sec:cantor} and let $\mathbb 1_{\bj}^\bi$ denote the
characteristic function of the set
\[
  \{\bo\in \Sigma\mid \exists p\in P(n(\bi))\text{ s.t. }(\sigma^p\bo)|_{\ell(p)}=\bj\text{ and }\bi<\bj\}.
\]
Note that since $p\in P(n(\bi))$, $\ell(p)=n(\bi)$.

We set $X_\bi^{(0)}=1$ for every $\bi\in\Sigma_*$. Then, for $n\in \bbN$, we define
$X_{\bi}^{(n+1)}$ inductively for
every $\bi\in\Sigma_*$ by
\[
  X_{\bi}^{(n+1)}=\sum_{|\bj|=n(\bi)}\lambda^{s_\bi}_{\sigma^{|\bi|}(\bj)}\cdot\mathbb 1^\bi_\bj\cdot X_\bj^{(n)}.
\]
In particular,
\[
  X_{\bi}^{(1)}=\sum_{|\bj|=n(\bi)}\lambda^{s_{\bi}}_{\sigma^{|\bi|}(\bj)}\cdot\mathbb 1^\bi_\bj.
\]
Note that by choice of $s_{\bi}$, we obtain $\bbE ( X_\bi^{(1)})=1$.

Let us now set up a filtration of sigma-algebras. Let
\[
  a_n^\bi=\max\{\max \ell^{-1}(n(\bj))+n(\bj)\mid \bj \in \Xi^\bi_{n-1}\}, \quad\text{ with }a_0^\bi=0.
\]
Then let $\cF_n^\bi$ be the sigma-algebra generated by the cylinders of length at most $a_n^\bi$ on
$\Sigma$, with $\cF_0^\bi$ the trivial sigma-algebra.

\begin{proposition}
  For each $\bi\in \Sigma_*$ the collection $(\cF_n^\bi)$ is a filtration, \ie $\cF_{n}^\bi\supset
  \cF_{n-1}^\bi$. Further,
  \[
    \cF_n^\bi = \bigcup_{|\bj|=n(\bi), \bi<\bj}\cF_{n-1}^\bj.
  \]
\end{proposition}

\begin{proof}
  The first claim follows immediately from $a_{n}^\bi> a_{n-1}^\bi$.

  For the second claim, we prove by induction that
  \[
    \Xi_{n}^\bi=\bigcup_{|\bj|=n(\bi), \bi<\bj}\Xi_{n-1}^\bj.
  \]
  The case $n=1$ is immediate from definitions. Assume the claim holds up to $n$. Then
  \begin{align*}
    \Xi_{n+1}^\bi
    & = \{\bj\mid |\bj|=n(\bi'), \bi'<\bj \text{ for some }\bi'\in \Xi_{n}^\bi\}\\
    &=\bigcup_{|\bj|=n(\bi), \bi<\bj}\{\bj'\mid n(\bi')=|\bj'|, \bi'<\bj'\text{ for
    some }\bi'\in \Xi^{\bj}_n\}\\
    &=\bigcup_{|\bj|=n(\bi), \bi<\bj}\Xi^{\bj}_n,
  \end{align*}
  by the induction hypothesis and the definition of $\Xi_n^\bj$. Hence,
  \[
    a^\bi_n=\max \{a_{n-1}^\bj\mid |\bj|=n(\bi), \bi < \bj\},
  \]
  which proves the claim.
\end{proof}

\begin{proposition}\label{thm:uniforms}
  For every $\epsilon_0>0$ there exists a $K\geq1$ such that for every constructor function with $n(\emptyset)\geq K$
  \[
    0 \leq s(\gamma)-s_{\bi}<\epsilon_0 \quad\text{ for all } \bi \in \Sigma_*.
  \]
  Moreover, there exists a $K\geq1$ such that for every constructor function with
  $n(\emptyset)\geq K$ we get that for all $\bi\in \Sigma_*$
  \begin{equation}
    \left( \frac{n(\bi)}{p_\bi e^{\gamma|\bi|}}\right)^{\frac{1}{n(\bi) - |\bi|}}\le \sum_{j=1}^m
    \lambda_j^{s_\bi}p_j e^{\gamma}\le  \left( \frac{5 n(\bi)}{p_\bi
    e^{\gamma|\bi|}}\right)^{\frac{1}{n(\bi) - |\bi|}}.
    \label{eq:goodSumBound}
  \end{equation}
\end{proposition}

\begin{proof}
  Let $\bi\in \Sigma_*$. By the definition of $s_\bi$
  and using the approximations \eqref{eq:mi} and \eqref{eq:repeat}, we obtain
  \[
    1=\sum_{|\ba|=n(\bi) - |\bi|}\lambda_{\ba}^{s_\bi}(1-(1-p_\bi p_{\ba})^{m(\bi)})\le
    \frac{2p_\bi e^{\gamma|\bi|}}{n(\bi)}\left ( \sum_{i=1}^m \lambda_i^{s_\bi}p_ie^{\gamma}\right)^{n(\bi) - |\bi|},
  \]
  which rearranges to
  \[
    \left( \frac{n(\bi)}{2p_\bi e^{\gamma|\bi|}}\right)^{\frac{1}{n(\bi) - |\bi|}} \le
    \sum_{i=1}^m \lambda_i^{s_\bi}p_i e^{\gamma}.
  \]
  By \cref{it:condni4},
  \[
    \left( \frac{n(\bi)}{2p_\bi e^{\gamma|\bi|}}\right)^{\frac{1}{n(\bi) - |\bi|}}> 1
  \]
  so that $s(\gamma) \geq s_{\bi}$.

  Let $\epsilon_\gamma$ be as in \cref{eq:Cramer}. Then denote
  \[
    \cG_\bi^{\epsilon_\gamma} :=\{\bj \in \Sigma_{n(\bi) - |\bi|}\mid p_\bj e^{\gamma |\bj|}\le
    (1-\epsilon_\gamma)^{n(\bi) - |\bi|}\}.
  \]
  We estimate,
  \begin{align}
    1&=\sum_{|\bj|=n(\bi)-|\bi|}\lambda^{s_\bi}_{\bj}(1-(1-p_{\bi}p_\bj)^{m(\bi)})\nonumber\\
     &\geq\lambda_{\max}^{(n(\bi)-|\bi|)(s_{\bi}-s(\gamma))}\sum_{\bj\in\cG_{\bi}^{\epsilon_\gamma}}\lambda_{\bj}^{s(\gamma)}(1-(1-p_{\bi}p_\bj)^{m(\bi)})\label{eq:loweroneminusboundstart}
     \\
     &\geq\lambda_{\max}^{(n(\bi)-|\bi|)(s_{\bi}-s(\gamma))}\sum_{\bj\in\cG_{\bi}^{\epsilon_\gamma}}\lambda_{\bj}^{s(\gamma)}m(\bi)p_{\bi}p_{\bj}\left(1-\frac{m(\bi)p_{\bi}p_{\bj}}{2}\right)\nonumber\\
     \intertext{by \cref{eq:taylor}, and}
     &\geq\lambda_{\max}^{(n(\bi)-|\bi|)(s_{\bi}-s(\gamma))}\sum_{\bj\in\cG_\bi^{\epsilon_\gamma}}\lambda_{\bj}^{s(\gamma)}\frac{e^{\gamma
       n(\bi)}}{2n(\bi)}p_{\bi}p_{\bj}\left(1-\frac{e^{\gamma n(\bi)}p_{\bi}p_{\bj}}{
       n(\bi)}\right)\nonumber\\
       \intertext{by the definition of $\cG^{\epsilon_\gamma}_{\bi}$. So,}
     &\geq\lambda_{\max}^{(n(\bi)-|\bi|)(s_{\bi}-s(\gamma))}\sum_{\bj\in\cG_\bi^{\epsilon_\gamma}}\lambda_{\bj}^{s(\gamma)}\frac{e^{\gamma
     n(\bi)}}{2n(\bi)}p_{\bi}p_{\bj}\left(1-\frac{e^{\gamma
       |\bi|}p_{\bi}(1-\epsilon_\gamma)^{n(\bi)-|\bi|}}{
   n(\bi)}\right)\nonumber
   \intertext{and by \cref{it:condni4},}
     &\geq\lambda_{\max}^{(n(\bi)-|\bi|)(s_{\bi}-s(\gamma))}\sum_{\bj\in\cG_\bi^{\epsilon_\gamma}}\lambda_{\bj}^{s(\gamma)}\frac{e^{\gamma
     n(\bi)}}{4n(\bi)}p_{\bi}p_{\bj}\nonumber\\
     &\geq\lambda_{\max}^{(n(\bi)-|\bi|)(s_{\bi}-s(\gamma))}\frac{e^{\gamma|\bi|}p_{\bi}}{4n(\bi)}\sum_{\bj\in\cG_\bi^{\epsilon_\gamma}}\lambda_{\bj}^{s(\gamma)}e^{\gamma
     (n(\bi)-|\bi|)}p_{\bj}\label{eq:loweroneminusbound}
  \end{align}
  The sum in \cref{eq:loweroneminusbound} is the $\bbQ_1$ measure of $\cG^{\epsilon_\gamma}_{\bi}$. Then by Cram\'er's theorem \cref{eq:Cramer}
  \[
    \sum_{\bj\in\left(\cG_\bi^{\epsilon_\gamma}\right)^c}\lambda_{\bj}^{s(\gamma)}e^{\gamma
    (n(\bi)-|\bi|)}p_{\bj}\leq e^{-\beta(n(\bi)-|\bi|)}.
  \]

  Now, taking logarithms in \cref{eq:loweroneminusbound} and dividing by
  $(n(\bi)-|\bi|)\log\lambda_{\max}$, we obtain
  \begin{equation*}
    (s_{\bi}-s(\gamma))
    +\frac{\gamma|\bi|+\log p_{\bi}-\log
    4n(\bi)+\log(1-e^{-\beta(n(\bi)-|\bi|)})}{(n(\bi)-|\bi|)\log \lambda_{\max}} \geq 0.
  \end{equation*}
  The second term is bounded above by
  \begin{align*}
    \frac{\gamma|\bi|+\log p_{\bi}-\log 4n(\bi)+\log(1-e^{-\beta(n(\bi)-|\bi|)})}{(n(\bi)-|\bi|)\log \lambda_{\max}}
&\leq
\frac{|\bi|\log p_{\min}-\log 4n(\bi)+\log(1-e^{-\beta n(\bi)/2})}{\tfrac12 n(\bi)\log \lambda_{\max}}\\
&\leq2\cdot\frac{\log p_{\min}-\log4n(\emptyset)+\log(1-e^{-\beta
n(\emptyset)/2})}{n(\emptyset)\log\lambda_{\max}},
  \end{align*}
  where we have used \cref{it:condni3}, \cref{it:condni5} and the monotonicity of the above maps.
  Hence, we may set $K$ large enough such that
  $s_{\bi}-s(\gamma) < \epsilon_0$ holds for our chosen $\epsilon_0$.

  We finish the proof by showing the upper bound in \cref{eq:goodSumBound}. Choose $\epsilon_0>0$
  such that $\lambda_{\min}^{-\epsilon_0}e^{-\beta}<1$, where $\beta$ is the exponent from
  \cref{eq:Cramer}. Then choose $K\geq1$ corresponding to $\epsilon_0>0$.  A computation similar
  to that from \cref{eq:loweroneminusboundstart} to \cref{eq:loweroneminusbound}
  gives
  \begin{align*}
    1
    &\geq
    \frac{e^{\gamma|\bi|}p_{\bi}}{4n(\bi)}\sum_{\bj\in\cG_\bi^{\epsilon_\gamma}}\lambda_{\bj}^{s_{\bi}}e^{\gamma
    (n(\bi)-|\bi|)}p_{\bj}\\
    &=
    \frac{e^{\gamma|\bi|}p_{\bi}}{4n(\bi)}
    \left(\sum_{\bj\in\Sigma_{n(\bi)-|\bi|}}\lambda_{\bj}^{s_{\bi}}e^{\gamma (n(\bi)-|\bi|)}p_{\bj}
      -\sum_{\bj\in(\cG_\bi^{\epsilon_\gamma})^c}\lambda_{\bj}^{s_{\bi}}e^{\gamma
    (n(\bi)-|\bi|)}p_{\bj}\right)\\
    &=\frac{e^{\gamma|\bi|}p_{\bi}}{4n(\bi)}\left(\left(\sum_{j=1}^N\lambda_{j}^{s_{\bi}}e^{\gamma
      }p_{j}\right)^{n(\bi)-|\bi|}
      -\sum_{\bj\in(\cG_\bi^{\epsilon_\gamma})^c}\lambda_{\bj}^{s_{\bi}}e^{\gamma
    (n(\bi)-|\bi|)}p_{\bj}\right)\\
    &=\frac{e^{\gamma|\bi|}p_{\bi}}{4n(\bi)}\left(\left(\sum_{j=1}^N\lambda_{j}^{s_{\bi}}e^{\gamma
      }p_{j}\right)^{n(\bi)-|\bi|}
      -\lambda_{\min}^{(s_\bi-s(\gamma))(n(\bi)-|\bi|)}\sum_{\bj\in(\cG_\bi^{\epsilon_\gamma})^c}\lambda_{\bj}^{s(\gamma)}e^{\gamma
    (n(\bi)-|\bi|)}p_{\bj}\right)
    \\
    &\geq
    \frac{e^{\gamma|\bi|}p_{\bi}}{4n(\bi)}\left(\left(\sum_{j=1}^N\lambda_{j}^{s_{\bi}}e^{\gamma
      }p_{j}\right)^{n(\bi)-|\bi|}
    -\lambda_{\min}^{(s_\bi-s(\gamma))(n(\bi)-|\bi|)}e^{-\beta(n(\bi)-|\bi|)}\right)
  \end{align*}
  by Cram\'er's theorem \cref{eq:Cramer}, as above.
  Now, rearranging gives
  \begin{align*}
    \sum_{j=1}^N\lambda_{j}^{s_{\bi}}e^{\gamma}p_{j}
     &\leq
     \left(
       \frac{4n(\bi)}{e^{\gamma|\bi|}p_{\bi}}+
       \lambda_{\min}^{(s_\bi-s(\gamma))(n(\bi)-|\bi|)}e^{-\beta(n(\bi)-|\bi|)}
     \right)^{1/(n(\bi)-|\bi|)}\\
     &\leq \left(
       \frac{5n(\bi)}{e^{\gamma|\bi|}p_{\bi}}
     \right)^{1/(n(\bi)-|\bi|)}.\qedhere
  \end{align*}  
\end{proof}

\begin{proposition}\label{thm:martingale}
  For all $\bi\in \Sigma_*$, the sequence of random variables $(X_\bi^{(n)})$ is a martingale with
  respect to the sequence of sigma-algebras $\cF_n^\bi$.
\end{proposition}

\begin{proof}
  We will argue by induction. Let $n=1$ and $\bi\in \Sigma_*$. Then, since $|\bi|\le \min \ell^{-1}(n(\bi))$,
  \[
    \bbE (X_\bi^{(1)}\mid \cF_0^\bi)=\bbE(X_\bi^{(1)})=1=X_\bi^{(0)}.
  \]
  Assume, inductively, that for all $\bi\in\Sigma_*$ and for $n$ the claim holds, that is,
  \[
    \bbE(X_\bi^{(n)}\mid \cF_{n-1}^\bi)=X_{\bi}^{(n-1)}.
  \]
  Since $a_1^\bi=\max\ell^{-1}(n(\bi))<a_{n}^\bi$, the random variable $\mathbb 1_{\bj}^\bi$ is
  measurable with respect to the sigma-algebra $\cF_n^{\bi}$, and hence
  \begin{align}
    \bbE(X_{\bi}^{(n+1)}\mid \cF_n^{\bi})
    &=\sum_{|\bj|=n(\bi),
    \bi<\bj}\lambda_{\sigma_{|\bi|}(\bj)}^{s_\bi}\bbE(\mathbb{1}^{\bi}_\bj X_{\bj}^{(n)}\mid
    \cF_n^{\bi})\nonumber\\
    &=\sum_{|\bj|=n(\bi), \bi<\bj}\lambda_{\sigma_{|\bi|}(\bj)}^{s_\bi}\mathbb{1}^{\bi}_\bj\bbE(
    X_{\bj}^{(n)}\mid \cF_n^{\bi}).\label{eq:mart}
  \end{align}
  Notice now that the random variable $X_\bj^{(n)}$ is independent of digits of $\bo$ beyond index
  $a_{n-1}^\bj$. Hence conditioning on $\cF_n^\bi$ is equivalent to conditioning on $\cF_{n-1}^\bi$, and
  \[
    \bbE(X_\bj^{(n)}\mid \cF_n^{\bi})=\bbE(X_\bj^{(n)}\mid \cF_{n-1}^{\bi})=X_\bj^{(n-1)},
  \]
  where we have made use of the induction hypothesis. Continuing with this from \eqref{eq:mart},
  \[
    \sum_{|\bj|=n(\bi), \bi<\bj}\lambda_{\sigma_{|\bi|}(\bj)}^{s_\bi}\mathbb 1^{\bi}_\bj\bbE(
    X_{\bj}^{(n)}\mid \cF_n^{\bi})=\sum_{|\bj|=n(\bi),
    \bi<\bj}\lambda_{\sigma_{|\bi|}(\bj)}^{s_\bi}\mathbb 1^{\bi}_\bj X_{\bj}^{(n-1)}=X_\bi^{(n)}.
  \]
  This proves the martingale property.
\end{proof}

\begin{proposition}\label{thm:L2}
  There exists a $K\geq1$ such that for every constructor function with $n(\emptyset)\geq K$ there
  is $c>0$ such that for all $n\in \bbN, \bi\in \Sigma_*$, we have
  \[
    \bbE((X_\bi^{(n)})^2)<c.
  \]
\end{proposition}
\begin{proof}
  We will find $\tilde \lambda<1$ such that for all $n\in \bbN$,
  \[
    \max_{\bi\in \Sigma_*}\bbE((X_\bi^{(n)})^2)\le 1+\max_{\bi\in \Sigma_*}\bbE((X_\bi^{(n-1)})^2)\tilde\lambda,
  \]
  from which it follows by induction that for all $n\in\bbN$,
  \[
    \max_{\bi\in \Sigma_*}\bbE((X_\bi^{(n)})^2)\le \frac{1}{1-\tilde\lambda}.
  \]

  Fix $\bi\in \Sigma_*$ and $n\in\bbN$. Then
  \begin{align}
    \bbE((X_\bi^{(n+1)})^2)
    &=\sum_{|\bj|=n(\bi),
    \bi<\bj}\lambda^{2s_\bi}_{\sigma^{|\bi|}(\bj)}\bbE(\mathbb{1}^{\bi}_\bj(X_\bj^{(n)})^2)\nonumber\\
    &\hspace{10em}+\sum_{\bj_1\neq \bj_2, |\bj_k|=n(\bi),
    \bi<\bj_k}\lambda^{s_\bi}_{\sigma^{|\bi|}(\bj_1)}\lambda^{s_\bi}_{\sigma^{|\bi|}(\bj_2)}
    \bbE(\mathbb 1^{\bi}_{\bj_1}\mathbb 1^{\bi}_{\bj_2}X_{\bj_1}^{(n)}X_{\bj_2}^{(n)})\label{eq:mart2}\\
    &=:S_1+S_2.\nonumber
  \end{align}
  Note that since $n(\bi)\neq n(\bj)$, $\mathbb 1^\bi_\bj$ is independent of $X_{\bj}^{(n)}$. Hence,
  $S_1$ in the sum above is equal to
  \[
    S_1 = \sum_{|\bj|=n(\bi), \bi<\bj}
    \lambda^{2s_\bi}_{\sigma^{|\bi|}(\bj)}(1-(1-p_\bj)^{m(\bi)})\bbE((X_{\bj}^{(n)})^2)\le
    \lambda_{\max}^{(n(\bi)-|\bi|)s_{\bi}}\max_{|\bj|=n(\bi), \bi<\bj}\bbE((X_{\bj}^{(n)})^2),
  \]
  where we have used the definition of $s_\bi$. Let us choose $\epsilon_0>0$ such that
  $s(\gamma)-\epsilon_0>0$. By \cref{thm:uniforms}, there exists a $K\geq1$ such that for every
  constructor function with $n(\emptyset)\geq K$, $s_\bi\geq s(\gamma)-\epsilon_0$ for every
  $\bi\in\Sigma_*$. Thus,
  $\lambda_{\max}^{s_\bi(n(\bi)-|\bi|)}\le\lambda_{\max}^{(s(\gamma)-\epsilon_0)}=:\tilde
  \lambda<1$, where we have used the fact that $n(\bi)-|\bi|\geq \frac12n(\bi)\geq
  \frac12n(\emptyset)\geq1$.

  For the second term in the sum \eqref{eq:mart2} we note first of all that since $n(\bi)\neq
  n(\bj_k)$,  $\mathbb 1^\bi_{\bj_1}\mathbb 1^\bi_{\bj_2}$ is independent of $X_{\bj_k}^{(n)}$. Since,
  furthermore, $\bbE(X_{\bj_k}^{(n)})=1$ by definition, we obtain for $S_2$
  \begin{align*}
    &\sum_{\bj_1\neq \bj_2, |\bj_k|=n(\bi),
    \bi<\bj_k}\lambda^{s_\bi}_{\sigma^{|\bi|}}(\bj_1)\lambda^{s_\bi}_{\sigma^{|\bi|}}(\bj_2)
    \bbE(\mathbb 1^{\bi}_{\bj_1}\mathbb 1^{\bi}_{\bj_2}X_{\bj_1}^{(n)}X_{\bj_2}^{(n)})\\[1.0em]
    &=\sum_{\bj_1\neq \bj_2, |\bj_k|=n(\bi),
    \bi<\bj_k}\lambda^{s_\bi}_{\sigma^{|\bi|}(\bj_1)}\lambda^{s_\bi}_{\sigma^{|\bi|}(\bj_2)}(1-(1-p_{\bj_1})^{m(\bi)}
    - (1-p_{\bj_2})^{m(\bi)} + (1-p_{\bj_2}-p_{\bj_2})^{m(\bi)}).
  \end{align*}
  Here, from the elementary computation
  \begin{align}
    1-p_{\bj_1} - p_{\bj_2}&\le 1-p_{\bj_1} - p_{\bj_2}+p_{\bj_1}p_{\bj_2}\Leftrightarrow\nonumber\\
    (1-p_{\bj_1} - p_{\bj_2})^{m(\bi)}&\le (1-p_{\bj_1})^{m(\bi)} (1- p_{\bj_2})^{m(\bi)}\Leftrightarrow\nonumber\\
    1-(1-p_{\bj_1})^{m(\bi)}-(1- p_{\bj_2})^{m(\bi)}+(1-p_{\bj_1} - p_{\bj_2})^{m(\bi)}&\le
    (1-(1-p_{\bj_1})^{m(\bi)}) (1-(1- p_{\bj_2})^{m(\bi)})\label{eq:poscor}
  \end{align}
  we can deduce that
  \[
    S_2\le \left(\sum_{|\bj|=n(\bi), \bi<\bj}\lambda^{s_\bi}_{\sigma^{|\bi|}(\bj)} (1-(1-p_{\bj})^{m(\bi)})\right)^2,
  \]
  which equals $1$ by definition of $s_\bi$. Thus, $S_1+S_2\le 1+\max\bbE((X_\bj^{(n)})^2)\tilde
  \lambda$, as in the claim.
\end{proof}

The upshot of the content of this section is the following: For each $\bi\in \Sigma_*$ there is a
limiting value $X_\bi$, which we will be able to use for the definition of the mass distribution on
the Cantor subset of $R(\bo, \ell)$ from \cref{sec:cantor}.

\begin{proposition}\label{thm:summability}
  For all $\bi\in \Sigma_*$, there is $X_\bi$ such that $X_\bi^{(n)}\to X_\bi$ as $n\to \infty$ almost
  surely and in $L^2$. Furthermore, there exists a $C>0$ such that $\bbE(X_{\bi}^2)\leq C$, $\bbE(X_\bi)=1$ and 
  \[
    X_\bi=\sum_{|\bj|=n(\bi), \bi<\bj} \lambda_{\sigma^{|\bi|}(\bj)}^{s_\bi}\mathbb 1_\bj^\bi X_\bj
  \]
  for all $\bi\in\Sigma_*$.
\end{proposition}
\begin{proof}
  This is immediate from \cref{thm:L2,thm:martingale} and Doob's $L^2$-martingale convergence
  theorem \cite[Theorem~VII.4.1]{Doob}.
\end{proof}

We will define a random measure $\mu=\mu_\bo$ on $\Sigma$ using the weights $X_\bi$ from the last
section, for cylinders of the following form: Consider $[\bi_1\bi_2\dots\bi_n]$ where
$n(\emptyset)=\bi_1$, $n(\bi_1)=\bi_1\bi_2$, and so on, up to $n(\bi_1\dots\bi_{n-1})=|\bi_1\dots
\bi_n|$. For any cylinder of this form, set
\begin{equation*}
  \mu([\bi_1\dots\bi_n]):=\lambda_{\bi_1}^{s_\emptyset}\lambda_{\bi_2}^{s_{{\bi_1}}}\cdots
  \lambda_{\bi_n}^{s_{\bi_1\dots\bi_{n-1}}}\mathbb 1_{\bi_1}^{\emptyset} \dots \mathbb
  1_{\bi_{1}\dots \bi_n}^{\bi_1\dots \bi_{n-1}}X_{\bi_1\dots \bi_n}.
\end{equation*}
By \cref{thm:summability} and Kolmogorov's extension theorem \cite[Theorem~1.7.8]{Tao}, for almost
every $\bo$ the measure $\mu$ is a well-defined Borel measure on $\Sigma$. By the construction, it
is supported on the random Cantor set defined in \cref{sec:cantor} and satisfies
$\mu(\Sigma)=\mu(R(\bo, \ell))=X_\emptyset$. Note that $\mu$ can be a zero measure, but using the
martingale property, $\bbE(\mu(\Sigma))=1$. For the remainder of this section we study properties
of $\mu$.

\subsection{Properties of the random measure on the Cantor set}\label{subsec:measure}

\begin{lemma}\label{lem:lemma6}
  There exists a $K\geq1$ such that for every constructor function with $n(\emptyset)\geq K$ we
  have that for every $\bi,\bj\in\Sigma_*$ with $|\bj|\leq n(\bi)-|\bi|$,
  $$
  \sum_{|\hbar|=n(\bi)-|\bi|-|\bj|}\lambda_{\hbar}^{s_\bi}\left(1-(1-p_{\bi\bj\hbar})^{m(\bi)}\right)\leq
  10p_{\bj}e^{\gamma|\bj|}.
  $$
\end{lemma}

\begin{proof}
  By \cref{eq:repeat} and \cref{eq:mi},
  \begin{align*}
    \sum_{|\hbar|=n(\bi)-|\bi|-|\bj|}\lambda_{\hbar}^{s_\bi}\left(1-(1-p_{\bi\bj\hbar})^{m(\bi)}\right)
      &\leq2\frac{p_\bi
      e^{\gamma|\bi|}}{n(\bi)}p_{\bj}e^{\gamma|\bj|}\sum_{|\hbar|=n(\bi)-|\bi|-|\bj|}\lambda_{\hbar}^{s_\bi}p_{\hbar}e^{\gamma|\hbar|}\\
      &=2\frac{p_\bi e^{\gamma|\bi|}}{n(\bi)}p_{\bj}e^{\gamma|\bj|}\left(\sum_{j=1}^m\lambda_{j}^{s_\bi}p_{j}e^{\gamma}\right)^{n(\bi)-|\bi|-|\bj|}\\
      \intertext{by \cref{thm:uniforms}, one can choose $K\geq1$ such that for every constructor function we get}
      &\leq 2\frac{p_\bi e^{\gamma|\bi|}}{n(\bi)}p_{\bj}e^{\gamma|\bj|}\left(\frac{5n(\bi)}{p_\bi
      e^{\gamma|\bi|}}\right)^{\frac{n(\bi)-|\bi|-|\bj|}{n(\bi)-|\bi|}}\\
      &\leq 10p_{\bj}e^{\gamma|\bj|}\left(\frac{p_\bi e^{\gamma|\bi|}}{n(\bi)}\right)^{\frac{|\bj|}{n(\bi)-|\bi|}}\\
      \intertext{by \cref{it:condni4}}
      &\leq 10p_{\bj}e^{\gamma|\bj|}.
  \end{align*}
\end{proof}

\begin{lemma}\label{lem:lemma7}
  For all $\epsilon_0>0$ there exists a $K\geq1$ such that for every constructor function with
  $n(\emptyset)\geq K$ the following holds: For all $\bj\in \Sigma_*$,
  \begin{equation*}
    \bbE(\mu([\bj]))\le C \lambda_{\min}^{-\epsilon_0 |\bj|}\bbQ_1([\bj]),
  \end{equation*}
  where $\bbQ_1$ is the Bernoulli measure with the probability vector $q_i=\lambda_i^{s(\gamma)}p_ie^\gamma$.
\end{lemma}

\begin{proof}
  Let $\bj\in\Sigma_*$ be arbitrary. Since $\{\Xi_\ell^\emptyset\}_{\ell\in\bbN}$ is a monotone
  refining sequence of partitions of $\Sigma$, there exists a maximal $n\in\bbN$ such that for
  every $k\leq n$ and every $\bi\in\Xi_k^\emptyset$, $[\bj]\cap[\bi]=\emptyset$ or
  $[\bj]\subsetneq[\bi]$. Thus, there exists $\bi_1,\ldots,\bi_n,\bj'\in\Sigma_*$ such that
  $\bj=\bi_1\dots\bi_n\bj'$ and $\bi_1\dots\bi_k\in\Xi_k^\emptyset$ for every $k$, moreover, for
  every $\hbar\in\Sigma_*$ with $|\hbar|=n(\bi_1\dots\bi_n)-|\bj|$,
  $\bj\hbar\in\Xi_{n+1}^\emptyset$. Hence,
  \begin{equation}\label{eq:muj}
    \mu([\bj])=\lambda_{\bi_1}^{s_\emptyset}\lambda_{\bi_2}^{s_{{\bi_1}}}\cdots
    \lambda_{\bi_n}^{s_{\bi_1\dots\bi_{n-1}}}\mathbb 1_{\bi_1}^{\emptyset} \dots \mathbb
    1_{\bi_{1}\dots \bi_n}^{\bi_1\dots
    \bi_{n-1}}\lambda_{\bj'}^{s_{\bi_1\dots\bi_n}}\sum_{|\hbar|=n(\bi_1\dots\bi_n)-|\bj|}\lambda_{\hbar}^{s_{\bi_1\dots\bi_n}}\mathbb
    1_{\bj\hbar}^{\bi_1\dots\bi_n}X_{\bj\hbar}.
  \end{equation}
  Then by the independence of the indicators and $X_{\bj\hbar}$,
  \begin{align*}
    \bbE(\mu([\bj]))&=\prod_{k=1}^{n}\lambda_{\bi_k}^{s_{\bi_1\dots\bi_{k-1}}}\left(1-(1-p_{\bi_1\dots\bi_k})^{m(\bi_1\dots\bi_{k-1})}\right)\\
		    &\hspace{10em}\cdot\lambda_{\bj'}^{s_{\bi_1\dots\bi_n}}\sum_{|\hbar|=n(\bi_1\dots\bi_n)-|\bj|}\lambda_{\hbar}^{s_{\bi_1\dots\bi_n}}\left(1-(1-p_{\bj\hbar})^{m(\bi)}\right).
  \end{align*}
  Choose $K\geq1$ sufficiently large such that \cref{thm:uniforms} and \cref{lem:lemma6} hold.
  Then by \cref{eq:repeat}, \cref{eq:mi} and \cref{lem:lemma6}
  \begin{align*}
    \bbE(\mu([\bj]))&\leq
    10\lambda_{\bi_1}^{s_\emptyset}\cdots\lambda_{\bi_n}^{s_{\bi_1\dots\bi_{n-1}}}p_{\bi_1\dots\bi_n}e^{\gamma|\bi_1\dots\bi_n|}\prod_{k=0}^{n-1}\frac{2p_{\bi_1\dots\bi_k}e^{\gamma|\bi_1\dots\bi_k|}}{n(\bi_1\dots\bi_k)}\lambda_{\bj'}^{s_{\bi_1\dots\bi_n}}p_{\bj'}e^{\gamma|\bj'|}\\
		    &\leq 10\lambda_{\min}^{-\epsilon_0|\bj|}\bbQ_1([\bj]),
  \end{align*}
  where in the last inequality we have used \cref{it:condni4} and \cref{thm:uniforms}.
\end{proof}

\begin{lemma}\label{lem:remainder}
  There exists a $K\geq1$ such that for every constructor function with $n(\emptyset)\geq K$ we
  have that for every $\bi,\bj\in\Sigma_*$ with $|\bj|\leq n(\bi)-|\bi|$
  $$
  \bbE\left(\left(\sum_{|\hbar|=n(\bi)-|\bi|-|\bj|}\lambda_{\hbar}^{s_{\bi}}\mathbb
  1_{\bi\bj\hbar}^{\bi}X_{\bi\bj\hbar}\right)^2\right)\leq
  10Cp_{\bj}e^{\gamma|\bj|}+100p_{\bj}^2e^{2\gamma|\bj|},
  $$
  where $C>0$ is the universal constant such that $\bbE(X_{\bi}^2)\leq C$ for all $\bi\in\Sigma_*$.
\end{lemma}

\begin{proof}
  Let $\bi,\bj\in\Sigma_*$ be fixed such that $|\bj|\leq n(\bi)-|\bi|$. Then
  \begin{align*}
		&\bbE\left(\left(\sum_{|\hbar|=n(\bi)-|\bi|-|\bj|}\lambda_{\hbar}^{s_{\bi}}\mathbb
		  1_{\bi\bj\hbar}^{\bi}X_{\bi\bj\hbar}\right)^2\right)=\sum_{\substack{|\hbar_1|=n(\bi)-|\bi|-|\bj|\\|\hbar_2|=n(\bi)-|\bi|-|\bj|}}\bbE\left(\lambda_{\hbar_1}^{s_{\bi}}\lambda_{\hbar_2}^{s_{\bi}}\mathbb
		  1_{\bi\bj\hbar_1}^{\bi}\mathbb
		1_{\bi\bj\hbar_2}^{\bi}X_{\bi\bj\hbar_1}X_{\bi\bj\hbar_2}\right)
		\intertext{by the independence of the indicators and $X_{\bi\bj\hbar}$'s}
		&=\sum_{|\hbar|=n(\bi)-|\bi|-|\bj|}\lambda_{\hbar}^{2s_{\bi}}\bbE(\mathbb
		1_{\bi\bj\hbar}^{\bi})\bbE(X_{\bi\bj\hbar}^2)+\sum_{\substack{|\hbar_1|=n(\bi)-|\bi|-|\bj|\\|\hbar_2|=n(\bi)-|\bi|-|\bj|\\\hbar_1\neq\hbar_2}}\lambda_{\hbar_1}^{s_{\bi}}\lambda_{\hbar_2}^{s_{\bi}}\bbE\left(\mathbb
		  1_{\bi\bj\hbar_1}^{\bi}\mathbb
		1_{\bi\bj\hbar_2}^{\bi}\right)\bbE(X_{\bi\bj\hbar_1})\bbE(X_{\bi\bj\hbar_1})
		\intertext{by \cref{thm:summability}}
		&\leq C\sum_{|\hbar|=n(\bi)-|\bi|-|\bj|}\lambda_{\hbar}^{2s_{\bi}}(1-(1-p_{\bi\bj\hbar})^{m(\bi)})\\
		&\qquad+\sum_{\substack{|\hbar_1|=n(\bi)-|\bi|-|\bj|\\|\hbar_2|=n(\bi)-|\bi|-|\bj|\\\hbar_1\neq\hbar_2}}\lambda_{\hbar_1}^{s_{\bi}}\lambda_{\hbar_2}^{s_{\bi}}(1-(1-p_{\bi\bj\hbar_1})^{m(\bi)}-(1-p_{\bi\bj\hbar_2})^{m(\bi)}+(1-p_{\bi\bj\hbar_1}-p_{\bi\bj\hbar_2})^{m(\bi)})
		\intertext{by \cref{eq:poscor}}
		&\leq C\sum_{|\hbar|=n(\bi)-|\bi|-|\bj|}\lambda_{\hbar}^{2s_{\bi}}(1-(1-p_{\bi\bj\hbar})^{m(\bi)})
		\\
		&\qquad\qquad\qquad+\sum_{\substack{|\hbar_1|=n(\bi)-|\bi|-|\bj|\\|\hbar_2|=n(\bi)-|\bi|-|\bj|\\\hbar_1\neq\hbar_2}}\lambda_{\hbar_1}^{s_{\bi}}\lambda_{\hbar_2}^{s_{\bi}}(1-(1-p_{\bi\bj\hbar_1})^{m(\bi)})(1-(1-p_{\bi\bj\hbar_1})^{m(\bi)})\\
		&\leq
		C\sum_{|\hbar|=n(\bi)-|\bi|-|\bj|}\lambda_{\hbar}^{2s_{\bi}}(1-(1-p_{\bi\bj\hbar})^{m(\bi)})+\left(\sum_{|\hbar|=n(\bi)-|\bi|-|\bj|}\lambda_{\hbar}^{s_{\bi}}(1-(1-p_{\bi\bj\hbar})^{m(\bi)})\right)^2
		\intertext{by \cref{lem:lemma6}}
		&\leq C10p_{\bj}e^{\gamma|\bj|}+\left(10p_{\bj}e^{\gamma|\bj|}\right)^2
  \end{align*}
  for every constructor function with sufficiently large $n(\emptyset)$.
\end{proof}

\begin{proposition}\label{thm:rightsubset}
  For every $\delta>0$ there exists a $K\geq1$ such that for every constructor function with
  $n(\emptyset)\geq K$ we have
  \[
    \mu_{\bo}\left( \bigcap_{k=0}^\infty \bigcup_{n=k}^\infty
      \bigcup_{\ell=\lceil\delta n\rceil}^\infty \{\bj\in
    \Sigma \mid p_{(\sigma^n\bj)|_{\ell}}e^{\gamma \ell}\ge (1-\epsilon_\gamma)^{\ell} \} \right)
    =0\quad\text{ for $\bbP$-almost every $\bo$},
  \]
  where $\epsilon_\gamma>0$ is defined in \cref{eq:Cramer}.
\end{proposition}

\begin{proof}
  Let $\delta>0$ be arbitrary but fixed. By Borel-Cantelli's Lemma, it is enough to show that
  \begin{equation*}
    \bbE\left(\sum_{n=1}^\infty\sum_{\ell=\lceil\delta n\rceil}^\infty \mu\left(\{\bj\in \Sigma
	  \mid p_{(\sigma^n\bj)|_{\ell}}e^{\gamma \ell}\ge (1-\epsilon_\gamma)^{\ell}
    \}\right)\right)<\infty.
  \end{equation*}
  Clearly,
  \[
    \mu\left(\{\bj\in \Sigma \mid p_{(\sigma^n\bj)|_{\ell}}e^{\gamma \ell}\ge
    (1-\epsilon_\gamma)^{\ell} \}\right)=
    \sum_{|\bi|=n}\sum_{|\bj|=\ell}\mu([\bi\bj])\mathbb{1}\{p_{\bj}
    e^{\gamma \ell}\ge (1-\epsilon_\gamma)^{\ell}\}.
  \]
  Let $\epsilon_0>0$ be such that $\lambda_{\min}^{-\epsilon_0(1+\delta)}e^{-\beta\delta}<1$,
  where $\beta$ is defined in \cref{eq:Cramer}. Note that
  $\lambda_{\min}^{-\epsilon_0}e^{-\beta}<1$. Let $K\geq1$ be such that \cref{lem:lemma7} holds
  with $\epsilon_0$. Hence, for every $n\geq N$
  \begin{align*}
    \bbE\left(\mu\left(\{\bj\in \Sigma \mid p_{(\sigma^n\bj)|_{\ell}}e^{\gamma \ell}\ge
	  (1-\epsilon_\gamma)^{\ell}
    \}\right)\right)
      &=\sum_{|\bi|=n}\sum_{|\bj|=\ell}\bbE(\mu([\bi\bj]))\mathbb{1}\{p_{\bj}e^{\gamma
      \ell}\ge (1-\epsilon_\gamma)^{\ell}\}\\
      &\leq
      \lambda_{\min}^{-\epsilon_0(n+\ell)}\sum_{|\bi|=n}\sum_{|\bj|=\ell}\bbQ_1([\bi\bj])\mathbb{1}\{p_{\bj}e^{\gamma
	\ell}\ge
      (1-\epsilon_\gamma)^{\ell}\}\\
      &=\lambda_{\min}^{-\epsilon_0(n+\ell)}\sum_{|\bj|=\ell}\bbQ_1([\bj])\mathbb{1}\{p_{\bj}e^{\gamma \ell}\ge (1-\epsilon_\gamma)^{\ell}\}\\
      &=\lambda_{\min}^{-\epsilon_0(n+\ell)}\bbQ_1\left(\{\bj\in \Sigma \mid p_{\bj|_{\ell}}e^{\gamma \ell}\ge (1-\epsilon_\gamma)^{\ell}\}\right)\\
      &\leq \lambda_{\min}^{-\epsilon_0(n+\ell)}e^{-\beta\ell}.
  \end{align*}
  Thus,
  \begin{align*}
    \sum_{\ell=\lceil\delta n\rceil}^\infty\bbE\left(\mu\left(\{\bj\in 
	  \Sigma \mid p_{(\sigma^n\bj)|_{\ell}}e^{\gamma \ell}\ge 
    (1-\epsilon_\gamma)^{\ell} \}\right)\right)\leq\sum_{\ell
  =\lceil\delta n\rceil}^\infty\lambda_{\min}^{-\epsilon_0(n+\ell)}
  e^{-\beta\ell}\leq\frac{\lambda_{\min}^{-\epsilon_0(1+\delta)n}
  e^{-\beta\delta n}}{1-\lambda_{\min}^{-\epsilon_0}e^{-\beta}},
\end{align*}
which forms a convergent series.
\end{proof}

For $\delta>0$, let
\[
  E_k^\delta:=\bigcap_{n=k}^\infty\bigcap_{\ell=\lceil n\delta\rceil}^\infty \{ \bj\in \Sigma\mid
  p_{(\sigma^n\bj)|_{\ell}}e^{\gamma \ell} < (1-\epsilon_\gamma )^{\ell}  \}
\]
It is an immediate consequence of \cref{thm:rightsubset}, $0<\bbE(\mu(\Sigma))<\infty$ and the
monotone convergence theorem that for a well-chosen constructor function
\[
  \bbE(\mu(E_k^\delta))>0
\]
for all large enough $k$. 

\begin{proposition}\label{thm:forfrostman}
  For all $\epsilon_0>0$ there exists a $\delta>0$ and a $K\geq1$ such that for every constructor
  function with $n(\emptyset)\geq K$ and every $q\in\bbN$:
  \begin{equation*}
    \sum_{\bj\in\Sigma_*}\lambda_{\bj}^{-s(\gamma)+3\epsilon_0}\bbE(\mu([\bj])^2)\mathbb{1}\{[\bj]\cap
    E_q^\delta\neq\emptyset\}<\infty.
  \end{equation*}

\end{proposition}

\begin{proof}
  Let $\bj\in\Sigma_*$ be arbitrary but fixed. Similarly to the proof of \cref{lem:lemma7}, there
  exists $\bi_1,\ldots,\bi_n,\bj'\in\Sigma_*$ such that $\bj=\bi_1\dots\bi_n\bj'$,
  $\bi_1\dots\bi_k\in\Xi_k^\emptyset$ for every $1\leq k\leq n$, and
  $\bj\hbar\in\Xi_{n+1}^\emptyset$ for every $\hbar\in\Sigma_*$ with
  $|\hbar|=n(\bi_1\dots\bi_n)-|\bj|$. By \cref{eq:muj}
  $$
  \mu([\bj])^2=\prod_{k=1}^n\lambda_{\bi_k}^{2s_{\bi_1\dots\bi_{k-1}}}\mathbb
  1_{\bi_{1}\dots \bi_k}^{\bi_1\dots
    \bi_{k-1}}\lambda_{\bj'}^{2s_{\bi_1\dots\bi_n}}\left(\sum_{|\hbar|=n(\bi_1\dots\bi_n)-|\bj|}\lambda_{\hbar}^{s_{\bi_1\dots\bi_n}}\mathbb
  1_{\bj\hbar}^{\bi_1\dots\bi_n}X_{\bj\hbar}\right)^2.
  $$
  By the independence of the indicators and $X_{\bj\hbar}$'s
  \begin{align*}
    \bbE(\mu([\bj])^2)&=\prod_{k=1}^{n}\lambda_{\bi_k}^{2s_{\bi_1\dots\bi_{k-1}}}\left(1-(1-p_{\bi_1\dots\bi_k})^{m(\bi_1\dots\bi_{k-1})}\right)\\
		      &\qquad\qquad\qquad\cdot\lambda_{\bj'}^{2s_{\bi_1\dots\bi_n}}\bbE\left(\left(\sum_{|\hbar|=n(\bi_1\dots\bi_n)-|\bj|}\lambda_{\hbar}^{s_{\bi_1\dots\bi_n}}\mathbb 1_{\bj\hbar}^{\bi_1\dots\bi_n}X_{\bj\hbar}\right)^2\right).
		      \intertext{By \cref{lem:remainder}}
		      &\leq\prod_{k=1}^{n}\lambda_{\bi_k}^{2s_{\bi_1\dots\bi_{k-1}}}\left(1-(1-p_{\bi_1\dots\bi_k})^{m(\bi_1\dots\bi_{k-1})}\right)\cdot\lambda_{\bj'}^{2s_{\bi_1\dots\bi_n}}\left(10Cp_{\bj'}e^{\gamma|\bj'|}+100p_{\bj'}^2e^{2\gamma|\bj'|}\right)
		      \intertext{By \cref{eq:repeat}, \cref{eq:mi} and \cref{it:condni4} and \cref{thm:uniforms}}
		      &\leq\prod_{k=1}^{n}\lambda_{\bi_k}^{2s_{\bi_1\dots\bi_{k-1}}}m(\bi_1\dots\bi_{k-1})p_{\bi_1\dots\bi_k}\cdot\lambda_{\bj'}^{2s_{\bi_1\dots\bi_n}}\left(10Cp_{\bj'}e^{\gamma|\bj'|}+100p_{\bj'}^2e^{2\gamma|\bj'|}\right)\\
		      &\leq\prod_{k=1}^{n}\lambda_{\bi_k}^{2s_{\bi_1\dots\bi_{k-1}}}\frac{2e^{\gamma n(\bi_1\dots\bi_{k-1})}p_{\bi_1\dots\bi_{k-1}}}{n(\bi\dots\bi_{k-1})}p_{\bi_k}\cdot\lambda_{\bj'}^{2s_{\bi_1\dots\bi_n}}\left(10Cp_{\bj'}e^{\gamma|\bj'|}+100p_{\bj'}^2e^{2\gamma|\bj'|}\right)\\
		      &\leq\prod_{k=1}^{n}\lambda_{\bi_k}^{2s_{\bi_1\dots\bi_{k-1}}}e^{\gamma|\bi_k|}p_{\bi_k}\cdot\lambda_{\bj'}^{2s_{\bi_1\dots\bi_n}}\left(10Cp_{\bj'}e^{\gamma|\bj'|}+100p_{\bj'}^2e^{2\gamma|\bj'|}\right)\\
		      &\leq10C\lambda_{\bj}^{s(\gamma)-2\epsilon_0}\bbQ_1([\bj])+100\lambda_{\bj}^{s(\gamma)-2\epsilon_0}\bbQ_1([\bi_1\dots\bi_n])\lambda_{\bj'}^{s(\gamma)}p_{\bj'}^2e^{2\gamma|\bj'|}.
  \end{align*}
  Let $A:=\max\{1,\sum_{j=1}^N\lambda_j^{s(\gamma)}p_j^2e^{2\gamma}\}$. Now, choose $\delta>0$ such that $\lambda_{\max}^{\epsilon_0/2}A^\delta<1$. By the construction of the sets $\Xi_n^\emptyset$,
  \begin{align*}
	&\sum_{\bj\in\Sigma_*}\lambda_{\bj}^{-s(\gamma)+3\epsilon_0}\bbE(\mu([\bj])^2)\mathbb{1}\{[\bj]\cap E_q^\delta\neq\emptyset\}\\
	&\qquad=\sum_{n=0}^\infty\sum_{\bi\in\Xi_n^\emptyset}\sum_{\ell=0}^{n(\bi)-|\bi|}\sum_{|\hbar|=\ell}\lambda_{\bi\hbar}^{-s(\gamma)+3\epsilon_0}\bbE(\mu([\bi\hbar])^2)\mathbb{1}\{[\bi\hbar]\cap E_q^\delta\neq\emptyset\}\\
	&\qquad\leq10C\sum_{\bj\in\Sigma_*}\lambda_{\bj}^{\epsilon_0}\bbQ_1([\bj])+100\sum_{n=0}^\infty\sum_{\bi\in\Xi_n^\emptyset}\sum_{\ell=0}^{n(\bi)-|\bi|}\sum_{|\hbar|=\ell}\lambda_{\bi\hbar}^{\epsilon_0}\bbQ_1([\bi])\lambda_{\hbar}^{s(\gamma)}p_{\hbar}^2e^{2\gamma|\hbar|}\mathbb{1}\{[\bi\hbar]\cap E_q^\delta\neq\emptyset\}.
  \end{align*}
  Since the first sum is clearly convergent, we only need to check the summability of the second part. For $n$ sufficiently large, $|\bi|\geq q$ for every $\bi\in\bigcup_{k=n}^\infty\Xi_k^\emptyset$ and $|\bi|\leq\lambda_{\max}^{-\epsilon_0/2|\bi|}$. So, for every sufficiently large $n$
  \begin{align*}
	&\sum_{\bi\in\Xi_n^\emptyset}\sum_{\ell=0}^{n(\bi)-|\bi|}\sum_{|\hbar|=\ell}\lambda_{\bi\hbar}^{\epsilon_0}\bbQ_1([\bi])\lambda_{\hbar}^{s(\gamma)}p_{\hbar}^2e^{2\gamma|\hbar|}\mathbb{1}\{[\bi\hbar]\cap E_q^\delta\neq\emptyset\}\\
	&\qquad=\sum_{\bi\in\Xi_n^\emptyset}\sum_{\ell=0}^{\lfloor\delta|\bi|\rfloor}\sum_{|\hbar|=\ell}\lambda_{\bi\hbar}^{\epsilon_0}\bbQ_1([\bi])\lambda_{\hbar}^{s(\gamma)}p_{\hbar}^2e^{2\gamma|\hbar|}\mathbb{1}\{[\bi\hbar]\cap E_q^\delta\neq\emptyset\}\\
	&\qquad\qquad\qquad\qquad\qquad\qquad+\sum_{\bi\in\Xi_n^\emptyset}\sum_{\ell=\lceil\delta|\bi|\rceil}^{n(\bi)-|\bi|}\sum_{|\hbar|=\ell}\lambda_{\bi\hbar}^{\epsilon_0}\bbQ_1([\bi])\lambda_{\hbar}^{s(\gamma)}p_{\hbar}^2e^{2\gamma|\hbar|}\mathbb{1}\{[\bi\hbar]\cap E_q^\delta\neq\emptyset\}\\
	&\qquad\leq\sum_{\bi\in\Xi_n^\emptyset}\lambda_{\bi}^{\epsilon_0}\bbQ_1([\bi])\sum_{\ell=0}^{\lfloor\delta|\bi|\rfloor}\left(\sum_{i=1}^N\lambda_{i}^{s(\gamma)}p_{i}^2e^{2\gamma}\right)^\ell+\sum_{\bi\in\Xi_n^\emptyset}\sum_{\ell=\lceil\delta|\bi|\rceil}^{n(\bi)-|\bi|}\sum_{|\hbar|=\ell}\lambda_{\bi\hbar}^{\epsilon_0}\bbQ_1([\bi])\lambda_{\hbar}^{s(\gamma)}p_{\hbar}e^{\gamma|\hbar|}(1-\epsilon_\gamma)^{|\hbar|}\\
	&\qquad\leq\sum_{\bi\in\Xi_n^\emptyset}\lambda_{\bi}^{\epsilon_0}\bbQ_1([\bi])\delta|\bi| A^{\delta|\bi|}+\sum_{\bi\in\Xi_n^\emptyset}\sum_{\ell=\lceil\delta|\bi|\rceil}^{n(\bi)-|\bi|}\lambda_{\bi}^{\epsilon_0}\bbQ_1([\bi])\lambda_{\max}^{\epsilon_0\ell}(1-\epsilon_\gamma)^{\ell}\\
	&\qquad\leq\sum_{\bi\in\Xi_n^\emptyset}\lambda_{\max}^{\epsilon_0/2|\bi|}\bbQ_1([\bi]) A^{\delta|\bi|}+\sum_{\bi\in\Xi_n^\emptyset}\lambda_{\max}^{\epsilon_0|\bi|}\bbQ_1([\bi])\left(1-\lambda_{\max}^{\epsilon_0}(1-\epsilon_\gamma)\right)^{-1}\\
	&\qquad\leq\lambda_{\max}^{\epsilon_0/2n}A^{\delta n}+\lambda_{\max}^{\epsilon_0n}\left(1-\lambda_{\max}^{\epsilon_0}(1-\epsilon_\gamma)\right)^{-1},
  \end{align*}
  which completes the proof.
\end{proof}

\subsection{Energy estimates of the random Cantor measure}\label{subsec:energy}

\begin{lemma}\label{thm:OSC}
  Let $\Phi=\{f_i(x)=\lambda_iO_ix+t_i\}_{i=1}^N$ be an IFS of similarities of $\bbR^d$ such that
  $\Phi$ satisfies the open set condition. For every Borel measure $\nu$ on
  $\Sigma=\{1,\ldots,N\}^\bbN$ with $0<\nu(\Sigma)<\infty$, if
  $\sum_{\bi\in\Sigma_*}\lambda_{\bi}^{-t}\nu([\bi])^2<\infty$ then $\dim_H\pi_*\nu\geq t$, where
  $\pi$ denotes the natural projection. 
\end{lemma}

The proof of \cref{thm:OSC} is standard, however, we could not find any proper reference for it.

\begin{proof}
  Without loss of generality, we may assume that $|U|=1$, where $U$ is the open and bounded subset
  of $\bbR^d$ with respect to the OSC holds. Let $\lambda_{\min}$ be the smallest contraction
  ratio as usual. For simplicity, let $\mu:=\pi_*\nu$. By Frostman's lemma
  \cite[Lemma~1.9.9]{BSS23}, it is enough to show that $\iint\|x-y\|^{-t}d\mu(x)d\mu(y)<\infty$.
  Then
  \begin{align*}
    \iint\frac{d\mu(x)d\mu(y)}{\|x-y\|^{t}}&=\sum_{n=0}^\infty\iint\|x-y\|^{-t}\mathbb{1}\{\lambda_{\min}^{n+1}<\|x-y\|\leq\lambda_{\min}^n\}d\mu(x)d\mu(y)\\
					   &\leq\sum_{n=0}^\infty\lambda_{\min}^{-t(n+1)}\int\mu\left(B(x,\lambda_{\min}^n)\right)d\mu(x).
  \end{align*}
  Let $$\cQ_n=\{\prod_{j=1}^d[m_j\lambda_{\min}^n,(m_j+1)\lambda_{\min}^n):m_j\in\bbZ\text{ for
  }j=1,\ldots,d\}$$ be the partition of $\bbR^d$ into axis parallel squares of side length
  $\lambda_{\min}^n$ such that the origin is a vertex of a partition element. Let $\cQ_n(x)$ be
  the unique partition element in $\cQ_n$ containing $x$. Furthermore, for
  $\cQ(x)=\prod_{j=1}^d[m_j\lambda_{\min}^n,(m_j+1)\lambda_{\min}^n)$, let
  $\widehat{\cQ}(x)=\prod_{j=1}^d[(m_j-1)\lambda_{\min}^n,(m_j+2)\lambda_{\min}^n)$. Then
  \begin{align*}
    \sum_{n=0}^\infty\lambda_{\min}^{-t(n+1)}\int\mu\left(B(x,\lambda_{\min}^n)\right)d\mu(x)&\leq\sum_{n=0}^\infty\lambda_{\min}^{-t(n+1)}\sum_{Q\in\cQ_n}\mu\left(\widehat{Q}\right)^2
  \end{align*}
  Let 
  $$
  \Gamma_n:=\{\bi\in\Sigma_*:\lambda_{\bi}\leq\lambda_{\min}^n<\lambda_{\bi_-}\}.
  $$
  Clearly, $\Gamma_n\cap\Gamma_m=\emptyset$ for every $n\neq m$. By the open set condition, there
  exists $C>0$ such that
  $\#\{\bi\in\Gamma_n:f_{\bi}(\overline{U})\cap\widehat{Q}\neq\emptyset\}\leq C$ for every
  $Q\in\cQ_n$. Thus,
  \begin{align*}
    \sum_{n=0}^\infty\lambda_{\min}^{-t(n+1)}\sum_{Q\in\cQ_n}\mu\left(\widehat{Q}\right)^2
    &\leq\sum_{n=0}^\infty\lambda_{\min}^{-t(n+1)}\sum_{Q\in\cQ_n}\left(\sum_{\substack{\bi\in\Gamma_n\\
    f_{\bi}(\overline{U})\cap\widehat{Q}\neq\emptyset}}\nu([\bi])\right)^2\\
    &\leq
    C^2\sum_{n=0}^\infty\lambda_{\min}^{-t(n+1)}\sum_{Q\in\cQ_n}\sum_{\substack{\bi\in\Gamma_n\\
    f_{\bi}(\overline{U})\cap\widehat{Q}\neq\emptyset}}\nu([\bi])^2\\
    &=C^2\sum_{n=0}^\infty\lambda_{\min}^{-t(n+1)}\sum_{\bi\in\Gamma_n}\#\{Q\in\cQ_n: f_{\bi}(\overline{U})\cap\widehat{Q}\neq\emptyset\}\nu([\bi])^2.
    \intertext{Since for every $\bi\in\Gamma_n$, there are at most $5^d$-many $Q\in\cQ_n$ such that $ f_{\bi}(\overline{U})\cap\widehat{Q}\neq\emptyset$}
    &\leq C^25^d\sum_{n=0}^\infty\lambda_{\min}^{-t(n+1)}\sum_{\bi\in\Gamma_n}\nu([\bi])^2\\
    &\leq C^25^d\lambda_{\min}^{-t}\sum_{n=0}^\infty\sum_{\bi\in\Gamma_n}\lambda_{\bi}^{-t}\nu([\bi])^2\\
    &\leq C^25^d\lambda_{\min}^{-t}\sum_{n=0}^\infty\sum_{\bi\in\Sigma_n}\lambda_{\bi}^{-t}\nu([\bi])^2.
  \end{align*}
\end{proof}

\begin{proposition}\label{thm:lowerbound}
  Let $\alpha>0$ be such that $\alpha\leq\sum_{i=1}^m \lambda_i^{s(\alpha)}p_ie^\alpha\log p_i$, where $\sum_{i=1}^m \lambda_i^{s(\alpha)}p_ie^\alpha=1$. Suppose that $\ell(n)$ satisfies
  $\liminf_{n\to\infty}\ell(n)/\log n = 1/\alpha$. Then, $\dim_H(R(\bo,\ell))
  \geq s(\alpha)$ for $\bbP$-almost every $\bo$.
\end{proposition}

\begin{proof}
  Clearly, the map $\alpha\mapsto s(\alpha)$ is continuous and monotone increasing. Thus, it is
  enough to show that $\bbP(\{\bo\in\Sigma:\dim_H(R(\bo,\ell))\geq s(\alpha-1/n)-1/n\})=1$ for
  every $n\in\bbN$. It is also clear that $R(\bo',\ell)=R(\bo,\ell)$ for every $\bo',\bo\in\Sigma$
  such that there exists $n\geq1$ such that $o_k'=o_k$ for every $k\geq n$. Thus,
  $\{\bo\in\Sigma:\dim_H(R(\bo,\ell))\geq s(\alpha-1/n)-1/n\}$ is a tail event and by Kolmogorov's
  zero-one law \cite[Corollary~1.7.1]{Khoshnevisan}, it is enough to show that
  $\bbP(\{\bo\in\Sigma:\dim_H(R(\bo,\ell))\geq s(\alpha-1/n)-1/n\})>0$ for every $n\in\bbN$.

  Let $\mu=\mu_{\bo}$ be the random measure constructed in \cref{sec:randommeasure} for
  $\alpha-1/n$. By \cref{thm:forfrostman}, choose $\delta>0$ and $K\geq1$ such that for
  $\epsilon_0=(3n)^{-1}$ we have
  \begin{equation}\label{eq:willbeenough}
    \sum_{\bj\in\Sigma_*}\lambda_{\bj}^{-s(\alpha-1/n)+1/n}\bbE(\mu([\bj])^2)\mathbb{1}\{[\bj]\cap
    E_q^\delta\neq\emptyset\}<\infty,
  \end{equation}
  for every $q\in\bbN$. By \cref{thm:summability} and \cref{thm:rightsubset}, there exists
  $q\in\bbN$ such that $\bbP(\{\bo\in\Sigma:\mu(E_q^\delta\cap R(\bo,\ell))>0\})>0$. Hence, it is
  enough to show that $\bbP(\{\bo\in\Sigma:\dim_H\mu|_{E_q^\delta}\geq s(\alpha-1/n)-1/n\})>0$. By
  \cref{thm:OSC}, it is enough to show that
  $$
  \bbE\left(\sum_{\bi\in\Sigma_*}\lambda_{\bi}^{-s(\alpha-1/n)+1/n}\mu|_{E_q^\delta}([\bi])^2\right)<\infty.
  $$
  But $\mu|_{E_q^\delta}([\bi])^2=\mu([\bi]\cap
  E_q^\delta)^2\leq\left(\mu([\bi])\mathbb{1}\{[\bi]\cap
  E_q^\delta\neq\emptyset\}\right)^2=\mu([\bi])^2\mathbb{1}\{[\bi]\cap E_q^\delta\neq\emptyset\}$,
  and so, the claim follows by \cref{eq:willbeenough}.
\end{proof}

\section{Completing the proofs}\label{sec:complete}

Finally, we summarize the proofs of the main theorems.

\begin{proof}[Proof of \cref{thm:main}]
  Let $\alpha>0$ and $\ell\colon\bbN\to\bbN$ be an increasing function such that
  $$
  \liminf_{n\to\infty}\frac{\ell(n)}{\log n}=\frac1\alpha,
  $$
  and let $\bbP_p$ be the Bernoulli measure defined by the probability vector
  $p=(p_1,\ldots,p_N)$. If $\alpha>\alpha_2$ then by \cref{thm:fullCover}, $R(\bo,\ell)=\Lambda$
  almost surely. If $\alpha_1<\alpha\leq \alpha_2$ then
  $\cH^{s_0}(R(\bo,\ell))=\cH^{s_0}(\Lambda)$ $\bbP_p$-almost surely by \cref{thm:notfullregion}.
  So we may assume that $\alpha\leq\alpha_1$.

  Let $s(\alpha)$ be as in \cref{eq:salpha}. By \cref{thm:twospec}, $s(\alpha)$ is the unique root
  of the probabilistic pressure defined in \cref{eq:probpressure}, and so, by \cref{thm:ubmain1},
  $\dim_HR(\bo,\ell)\leq s(\alpha)$ $\bbP_p$-almost surely. If $\alpha_0<\alpha<\alpha_1$ then
  $\dim_HR(\bo,\ell)\geq s(\alpha)$ by \cref{thm:spectrumregion} $\bbP$-almost surely. If
  $\alpha=\alpha_1$ then by taking $c>1$ arbitrary and having $\ell'(n):=c\ell(n)$, we get that
  $\dim_HR(\bo,\ell)\geq\dim_HR(\bo,\ell')=s(\alpha/c)$ $\bbP_p$-almost surely. By taking the
  limit $c\to1$ on a countable sequence and using the continuity of $\alpha\mapsto s(\alpha)$ (see
  \cref{thm:legendre}) we get that $\dim_HR(\bo,\ell)\geq\lim_{c\to1}s(\alpha/c)=s(\alpha)=s_0$
  $\bbP$-almost surely.

  Finally, suppose that $\alpha\leq\alpha_0$. Then, the  almost sure lower bound $\dim_H R(\bo,
  \ell)\ge s(\alpha)$ follows by \cref{thm:lowerbound}. 
\end{proof}

\begin{proof}[Proof of \cref{thm:main2}]
  The proof is the combination of \cref{thm:ubmain2}, \cref{thm:fullCover},
  \cref{thm:notfullregion} and \cref{thm:spectrumregion} for every $\alpha\neq\alpha_1$. One can
  verify the case $\alpha=\alpha_1$ by relying on the continuity of $\alpha\mapsto t(\alpha)$ and
  on the monotonicity of the set $R(\bo,\ell)$ similarly to the previous proof. 
\end{proof}

\bibliographystyle{alpha}
\bibliography{Bibliography}

\begin{thebibliography}{JJMS25}

\bibitem[AB21]{AllenBarany2021}
Demi Allen and Bal\'azs B\'ar\'any.
\newblock On the {H}ausdorff measure of shrinking target sets on self-conformal
  sets.
\newblock {\em Mathematika}, 67(4):807--839, 2021.

\bibitem[ABB25]{AllenBakerBarany2025}
Demi Allen, Simon Baker, and Bal\'azs B\'ar\'any.
\newblock Recurrence rates for shifts of finite type.
\newblock {\em Adv. Math.}, 460:Paper No. 110039, 36, 2025.

\bibitem[Bak23]{Baker2023}
Simon Baker.
\newblock Overlapping iterated function systems from the perspective of metric
  number theory.
\newblock {\em Mem. Amer. Math. Soc.}, 287(1428):v+95, 2023.

\bibitem[Bak24]{Baker2024}
Simon Baker.
\newblock Intrinsic {D}iophantine approximation for overlapping iterated
  function systems.
\newblock {\em Math. Ann.}, 388(3):3259--3297, 2024.

\bibitem[BG92]{BandtGraf1992}
Christoph Bandt and Siegfried Graf.
\newblock Self-similar sets. {VII}. {A} characterization of self-similar
  fractals with positive {H}ausdorff measure.
\newblock {\em Proc. Amer. Math. Soc.}, 114(4):995--1001, 1992.

\bibitem[BK24]{BakerKoivusalo2024}
Simon Baker and Henna Koivusalo.
\newblock Quantitative recurrence and the shrinking target problem for
  overlapping iterated function systems.
\newblock {\em Adv. Math.}, 442:Paper No. 109538, 65, 2024.

\bibitem[BR18]{BaranyRams2018}
Bal\'azs B\'ar\'any and Micha\l{} Rams.
\newblock Shrinking targets on {B}edford-{M}c{M}ullen carpets.
\newblock {\em Proc. Lond. Math. Soc. (3)}, 117(5):951--995, 2018.

\bibitem[BSS23]{BSS23}
Bal\'azs B\'ar\'any, K\'aroly Simon, and Boris Solomyak.
\newblock {\em Self-similar and self-affine sets and measures}, volume 276 of
  {\em Mathematical Surveys and Monographs}.
\newblock American Mathematical Society, Providence, RI, [2023] \copyright
  2023.

\bibitem[BT22]{BaranyTroscheit2022}
Bal\'azs B\'ar\'any and Sascha Troscheit.
\newblock Dynamically defined subsets of generic self-affine sets.
\newblock {\em Nonlinearity}, 35(10):4986--5013, 2022.

\bibitem[Dav25]{Daviaud2025}
EDOUARD Daviaud.
\newblock Dynamical {D}iophantine approximation and shrinking targets for
  {$C^1$} weakly conformal {IFS}s with overlaps.
\newblock {\em Ergodic Theory Dynam. Systems}, 45(6):1777--1826, 2025.

\bibitem[Doo90]{Doob}
J.~L. Doob.
\newblock {\em Stochastic processes}.
\newblock Wiley Classics Library. John Wiley \& Sons, Inc., New York, 1990.
\newblock Reprint of the 1953 original, A Wiley-Interscience Publication.

\bibitem[Dvo56]{Dvoretzky1956}
Aryeh Dvoretzky.
\newblock On covering a circle by randomly placed arcs.
\newblock {\em Proc. Nat. Acad. Sci. U.S.A.}, 42:199--203, 1956.

\bibitem[FST13]{FanSchmelingTroubetzkoy2013}
Ai-Hua Fan, J\"org Schmeling, and Serge Troubetzkoy.
\newblock A multifractal mass transference principle for {G}ibbs measures with
  applications to dynamical {D}iophantine approximation.
\newblock {\em Proc. Lond. Math. Soc. (3)}, 107(5):1173--1219, 2013.

\bibitem[Hut81]{Hutchinson1981}
John~E. Hutchinson.
\newblock Fractals and self-similarity.
\newblock {\em Indiana Univ. Math. J.}, 30(5):713--747, 1981.

\bibitem[HV95]{HillVelani1995}
Richard Hill and Sanju~L. Velani.
\newblock The ergodic theory of shrinking targets.
\newblock {\em Invent. Math.}, 119(1):175--198, 1995.

\bibitem[HV99]{HillVelani1999}
Richard Hill and Sanju~L. Velani.
\newblock The shrinking target problem for matrix transformations of tori.
\newblock {\em J. London Math. Soc. (2)}, 60(2):381--398, 1999.

\bibitem[JJMS25]{JarvenpaaJarvenpaaMyllyojaStenflo2025}
Esa J\"arvenp\"a\"a, Maarit J\"arvenp\"a\"a, Markus Myllyoja, and \"Orjan
  Stenflo.
\newblock The {E}kstr\"om-{P}ersson conjecture regarding random covering sets.
\newblock {\em J. Lond. Math. Soc. (2)}, 111(1):Paper No. e70058, 27, 2025.

\bibitem[JMS25]{JMS25}
Esa Järvenpää, Markus Myllyoja, and Stéphane Seuret.
\newblock Hitting probabilities and the ekstr{\"o}m-persson conjecture.
\newblock preprint, available at arXiv:2506.10448, 2025.

\bibitem[Kho02]{Khoshnevisan}
Davar Khoshnevisan.
\newblock {\em Multiparameter processes}.
\newblock Springer Monographs in Mathematics. Springer-Verlag, New York, 2002.
\newblock An introduction to random fields.

\bibitem[KKP23]{KirsebomKundePersson2023}
Maxim Kirsebom, Philipp Kunde, and Tomas Persson.
\newblock On shrinking targets and self-returning points.
\newblock {\em Ann. Sc. Norm. Super. Pisa Cl. Sci. (5)}, 24(3):1499--1535,
  2023.

\bibitem[KR18]{KoivusaloRamirez2018}
Henna Koivusalo and Felipe~A. Ram\'irez.
\newblock Recurrence to shrinking targets on typical self-affine fractals.
\newblock {\em Proc. Edinb. Math. Soc. (2)}, 61(2):387--400, 2018.

\bibitem[LS13]{LiaoSeuret2013}
Lingmin Liao and St\'ephane Seuret.
\newblock Diophantine approximation by orbits of expanding {M}arkov maps.
\newblock {\em Ergodic Theory Dynam. Systems}, 33(2):585--608, 2013.

\bibitem[PR17]{PerssonRams2017}
Tomas Persson and Micha\l{} Rams.
\newblock On shrinking targets for piecewise expanding interval maps.
\newblock {\em Ergodic Theory Dynam. Systems}, 37(2):646--663, 2017.

\bibitem[Sch94]{Schief1994}
Andreas Schief.
\newblock Separation properties for self-similar sets.
\newblock {\em Proc. Amer. Math. Soc.}, 122(1):111--115, 1994.

\bibitem[Seu18]{Seuret2018}
St\'ephane Seuret.
\newblock Inhomogeneous random coverings of topological {M}arkov shifts.
\newblock {\em Math. Proc. Cambridge Philos. Soc.}, 165(2):341--357, 2018.

\bibitem[Tao11]{Tao}
Terence Tao.
\newblock {\em An introduction to measure theory}, volume 126 of {\em Graduate
  Studies in Mathematics}.
\newblock American Mathematical Society, Providence, RI, 2011.

\end{thebibliography}

\Addresses

\end{document}